\documentclass{amsart}
\usepackage{amsmath,amssymb}
\usepackage{comment}
\usepackage{color}
\usepackage{graphicx,hyperref}
\usepackage{a4wide}
\usepackage{wasysym}

\newtheorem{thm}{Theorem}
\newtheorem{lem}{Lemma}
\newtheorem{pro}{Proposition}
\newtheorem{cor}{Corollary}

\newcommand{\bt}{{\textcolor{blue}{\scriptstyle\blacktriangle}}}
\newcommand{\sbt}{{\textcolor{blue}{\tiny\blacktriangle}}}

\newcommand{\bl}{{\textcolor{red}{\scriptstyle\blacklozenge}}}
\newcommand{\sbl}{{\textcolor{red}{\tiny\blacklozenge}}}

\newcommand{\ws}{{\scriptstyle\square}}

\newcommand{\bc}{{\textstyle\bullet}}\newcommand{\sbc}{{\scriptstyle\bullet}}
\newcommand{\wc}{{\textstyle\circ}}

\newcommand{\re}{\textcolor{red}{\bl\!\rule[1.5pt]{10pt}{1.7pt}}\!\ws}
\newcommand{\pe}{\bl\!\textcolor{blue}{\rule[1.5pt]{10pt}{1.7pt}}\!\bt}
\newcommand{\be}{\bc\!\textcolor{black}{\rule[1.5pt]{10pt}{1.7pt}}\!\ws}
\newcommand{\bcedge}{\bc\!\!\rule[2pt]{10pt}{1pt}}
\newcommand{\wsedge}{\ws\!\rule[2pt]{8pt}{1pt}}
\newcommand{\bledge}{\bl\!\!\textcolor{blue}{\rule[2pt]{10pt}{1pt}}}
\newcommand{\btedge}{\bt\!\!\textcolor{blue}{\rule[2pt]{10pt}{1pt}}}
\newcommand{\cQ}{\mathcal{Q}}

\newenvironment{descriptioncompact}{\begin{description} }{\end{description}}
\usepackage[backend=bibtex]{biblatex}
\begin{document}
\title[Order One Catalytic Equations and context-free specifications]{From order one catalytic decompositions to context-free specifications: the rewiring bijection}

\author{Enrica Duchi}
\thanks{\href{mailto:Enrica.Duchi@irif.fr}{Enrica.Duchi@irif.fr}.  ED was partially supported by ANR Projects IsOMa (Grant number ANR-21-CE48-0007) and CartesEtPlus (Grant number ANR-23-CE48-0018).}
\author{Gilles Schaeffer}
\thanks{\href{mailto:Gilles.Schaeffer@lix.polytechnique.fr}{Gilles.Schaeffer@lix.polytechnique.fr}.  GS was partially supported by ANR Project 3DMaps (Grant number ANR-20-CE48-0018), LambdaComb (Grant number ANR-21-CE48-0017) and CartesEtPlus (Grant number ANR-23-CE48-0018).}

\address{\textbf{E.D.} IRIF, Université Paris Cité}
\address{\textbf{G.S.} LIX, CNRS, Ecole Polytechnique, Institut Polytechnique de Paris}

\begin{abstract}{A celebrated result of Bousquet-Mélou and Jehanne states
  that the bivariate power series solutions of so-called
  \emph{discrete differential equations}, also known as
  \emph{polynomial equations with one catalytic variable}, or simply
  \emph{catalytic equations}, are algebraic series.  We give a purely
  combinatorial derivation of this result in the case of \emph{order
  one catalytic equations}, involving only one univariate
  unknown series.  In particular, our approach provides a tool to
  produce context-free specifications or bijections with simple
  multi-type families of trees for the derivation trees of
  combinatorial structures that are directly governed by an order one
  catalytic decomposition.

This provides a simple unified framework to deal with various
combinatorial interpretation problems that were solved or raised over
the last 50 years since the first such catalytic equation was written
by W. T. Tutte in the late 60's to enumerate rooted planar maps.}
\end{abstract}
\maketitle

\section{Introduction}
\subsection{Order one catalytic equations.} We are interested in a family of functional equations that have surfaced in various places in the mathematical literature since the 60s: an \emph{order one catalytic equation} is an equation of the form
\begin{equation}\label{eqn:eq1}
F(t,u)=t\cdot Q\left(F(t,u),\frac1u(F(t,u)-F(t,0)),u\right),
\end{equation}
with $Q(v,w,u)$ a given series in the ring of formal power series  $\mathbb{Q}[[v,w,u]]$ 
and $F(t,u)$ an unknown formal power series in
$\mathbb{Q}[[t,u]]$.
Equation~\eqref{eqn:eq1} is a special instance of the more general
family of \emph{equations with one catalytic variable} thoroughly discussed in
\cite{BMJ}, and also known as \emph{discrete differential equations}
as they are often written as above in terms of iterates of a divided difference
operator $\Delta F=\frac1u (F(t,u)-F(t,0))$ in the so-called
\emph{catalytic variable} $u$. These equations have gained over the
years the status of standard ingredients in enumerative combinatorics,
as discussed for instance in
\cite{BCNSED,gs:BM-ICM,BMJ,DNY,notar2024}.


The formal power series solutions of
order one catalytic equations are particularly well understood: as
discussed in \cite{S23}, it easily follows from the celebrated approach of  
Bousquet-Mélou and Jehanne \cite{BMJ} that the univariate part
$f\equiv f(t)=F(t,0)$ of the solution of Equation~\eqref{eqn:eq1}
admits generically the following simple univariate parametrization:
\begin{align}
  f\;=\;C_\ws-C_\bl\cdot C_\bt, \quad\textrm{ or } \quad
  {\textstyle \frac{d}{dt}}f&=(1+C_\bc)\cdot Q(C_\ws,C_\bt,C_\bl),
  \label{eq:fCFS}
\end{align}
where $C_\ws\equiv C_\ws(t)$, $C_\bc\equiv C_\bc(t)$, $C_\bl\equiv C_\bl(t)$ and $C_\bt\equiv C_\bt(t)$ are the unique formal power series that satisfy the companion system
\begin{align}
  \left\{
 \begin{array}{rclcl}
    {C}_\ws
    &=&t\cdot Q({C}_\ws,{C}_\bt,{C}_\bl),\\
    {C}_\bc
    &=&t\cdot (1+{C}_\bc)\cdot {Q}'_v ({C}_\ws,{C}_\bt,{C}_\bl),\\
    {C}_\bl
    &=&t\cdot (1+{C}_\bc)\cdot{Q}'_w ({C}_\ws,{C}_\bt,{C}_\bl),\\
    {C}_\bt
    &=&t\cdot (1+{C}_\bc)\cdot {Q}'_u ({C}_\ws,{C}_\bt,{C}_\bl).
   \end{array}
  \right.
  \label{eq:CFS}
\end{align}
When $Q$ is a polynomial, this parametrization immediately implies
that $f$ is an algebraic series: the fact that solutions of polynomial
equations with one catalytic variable are algebraic series follows in
fact from earlier general algebraic results (\cite{Popescu}, see the
discussion in \cite[Sec. 5.3]{Bonichon}), but the explicit dependency
of Parametrization \eqref{eq:CFS} in $Q$ is remarkably appealing
from a combinatorial point of view.  In particular, when $Q$ is a
polynomial with non-negative integer coefficients, this
parametrization shows that the series $\frac{d}{dt}f$ is
$\mathbb{N}$-algebraic in the sense of \cite{gs:BM-ICM}, and since the
first occurrences of catalytic equations in the literature, this fact
has called for combinatorial interpretations. Elegant interpretations
in terms of trees were indeed given for several special instances of
these equations (see \cite{fang-fusy-nadeau} and references therein,
as well as the conclusion of the present article), based on remarkable
ad-hoc constructions building on specific properties of the underlying
combinatorial structures.

The purpose of this article, as presented in the next section, is
instead to give a generic combinatorial interpretation and derivation
of Equations~\eqref{eq:fCFS} and~\eqref{eq:CFS} directly starting from
Equation~\eqref{eqn:eq1}. On the one hand this approach provides new
insights on algebraic combinatorial generating functions as discussed
by Bousquet-Mélou in \cite{gs:BM-ICM}, and on the other hand it
yields a systematic approach to build bijections with simple
families of trees for combinatorial structures with order one
catalytic specification.

\subsection{Two combinatorial interpretations and a bijective derivation}

When $Q(v,w,u)$ is a polynomial with non-negative
integer coefficients, it is not difficult to give combinatorial
interpretations to Equation~\eqref{eqn:eq1} in terms of labeled trees
with non-negativity conditions on labels. This was done for instance
in \cite{cori-schaeffer} for a closely related family of equations, in
terms of some \emph{description trees}, or more recently in
\cite{Chen} for a special case of Equation~\eqref{eqn:eq1} in terms of
some \emph{fully parked trees}.

Under the same hypotheses on $Q(v,w,u)$,
System~\eqref{eq:CFS} is a so-called \emph{$\mathbb{N}$-algebraic
system}, and the power series $C_\ws$, $C_\bc$, $C_\bl$ and $C_\bt$
admit natural interpretations as generating functions (gf) of \emph{simple
varieties of multi-type trees} with a \emph{context-free
specification}, as discussed in \cite{gs:BM-ICM}, or \cite[Chapter I,
  ex. I.53, p82]{FS}.  However there is \emph{a priori} no clear relations
between these first and second types of interpretations.
\begin{figure}
  \centerline{\includegraphics[scale=.4]{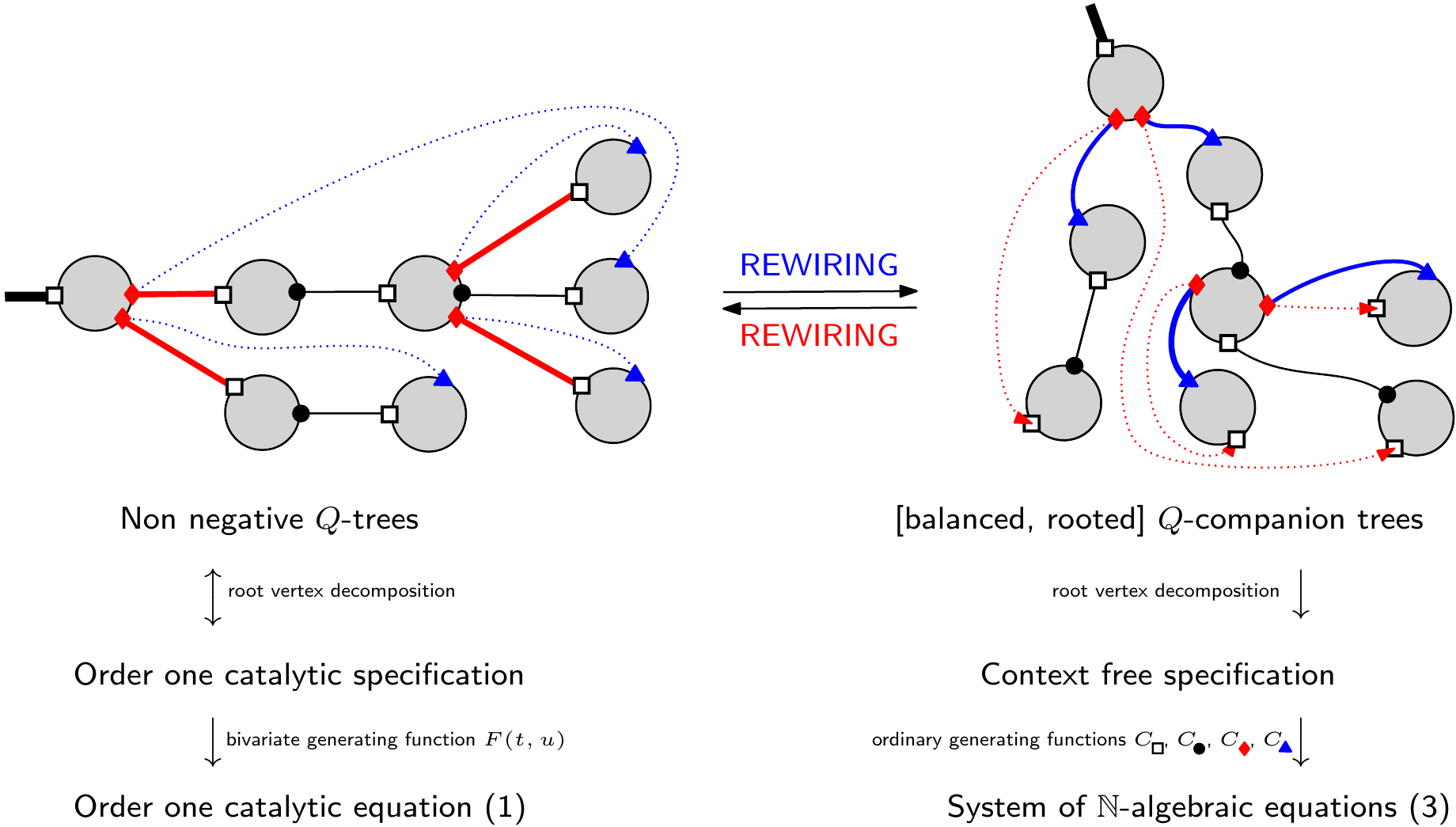}}
\caption{From order one catalytic specifications to context-free specifications, via rewiring non-negative $\cQ$-trees into $\cQ$-companion trees.\label{fig:plan}}
\end{figure}

Our main contribution is instead to introduce two families of plane
trees with vertices and edges decorated with pearls ($\ws$, $\bl$,
$\bc$, $\bt$), that give two such interpretations that can easily be
related one to the other, as illustrated by Figure~\ref{fig:plan}. The
trees of the first family, called \emph{non-negative $\cQ$-trees},
have two types of edges ($\be$ edges and $\re$ edges) and some free
$\bt$-pearls, and they obey a non-negativity condition on the
difference between the number of $\bt$-pearls and the number of
$\bl$-pearls in their subtrees: as stated in
Proposition~\ref{pro:cata}, these trees provide a generic
interpretation of Equation~\eqref{eqn:eq1}. The trees of the second
family, called \emph{rooted $\cQ$-companion trees}, have instead edges
of type $\be$ and $\pe$ and admit, as stated in Theorem~\ref{pro:CFS},
a context-free specification that interprets System~\ref{eq:CFS}.  The
relation between the two is given a simple transformation, called
\emph{rewiring}, that simply consists in replacing each $\re$ edge by
a well chosen $\pe$ edge. More precisely, as stated in
Theorem~\ref{thm:main}, rewiring yields a one-to-one mapping between
non-negative $\cQ$-trees and a certain subset of \emph{balanced}
$\cQ$-companion trees. An analysis of the balance condition,
summarized in Theorem~\ref{thm:un-balanced}, shows that these balanced
$\cQ$-companion trees are in one-to-one correspondence with unrooted
$\cQ$-companion trees, while unbalanced $\cQ$-companion trees can be
decomposed into pairs of $\bl$- and $\bt$-rooted $\cQ$-companion
trees. In particular the latter decomposition allows us to give in
Corollary~\ref{cor:F0} an interpretation of the first formula
in~\eqref{eq:fCFS}. Upon marking a vertex, balanced $\cQ$-companion
trees can instead be decomposed into pairs of $\bc$- and $\ws$-rooted
$\cQ$-companion trees, and, as discussed in Corollary~\ref{cor:F0O},
this yields a direct interpretation of the second Formula
in~\eqref{eq:fCFS}.

Let us conclude this introduction with the organization of the
article. In Section~\ref{sec:Qtrees} we first concentrate on non-negative $\cQ$-trees: after stating their definition and showing their
relation to Equation~\ref{eqn:eq1} (Section~\ref{sec:defQtrees}), we
recall some standard combinatorial results about spanning trees in
planar maps and balanced parenthesis strings (Section~\ref{sec:map})
that allow us to describe the rewiring transformation
(Section~\ref{sec:rewiring}) and start its analysis. In
Section~\ref{sec:Qcompanion} we introduce and study $\cQ$-companion
trees. It is readily clear that the image of a non-negative $\cQ$-tree
by rewiring is a balanced $\cQ$-companion tree, and the rest
of Section~\ref{sec:defQcompanion} is dedicated to the analysis of the
inverse transformation to prove the main result Theorem~\ref{thm:main}. The balance condition is analyzed in Section~\ref{sec:unrooted} and the relation of rooted $\cQ$-companion trees with System~\ref{eq:CFS} is discussed in Section~\ref{sec:decQcompanion}. As a corollary an alternative interpretation of the series in System~\ref{eq:CFS} in terms of marked non-negative $\cQ$-trees is sketched in Section~\ref{sec:markedQtrees}. Finally some examples, special cases and a slight generalization are presented in Section~\ref{sec:appli}, while a conclusion is drawn in Section~\ref{sec:Conclusion}.

\section{Non-negative $\cQ$-trees}\label{sec:Qtrees}
\subsection{Necklaces and non-negative $\cQ$-trees}\label{sec:defQtrees}
Let $\cQ$ denote a set of words on an alphabet $\{\bc,\bl,\bt\}$ of
\emph{pearls}: we identify each element $w=w_1\ldots w_k$ of $\cQ$
with a clockwise oriented necklace carrying one $\ws$-pearl followed
by the pearls $w_1,\ldots,w_k$. In the rest of the article the set
$\cQ$ will be viewed as the set of allowed vertex types for various
families of plane trees. Accordingly, to a set $\cQ$ of necklaces we
associate
the vertex type generating function $Q(v,w,u)$ as
\begin{equation}\label{eq:Q}
Q(v,w,u)=\sum_{s\in \cQ} v^{|s|_{\sbc}}w^{|s|_{\sbl}}u^{|s|_{\sbt}}.
\end{equation}
The necklace associated to the word $w=\bc\bl\bt\bc\bc\bl$ is
represented on Figure~\ref{fig:Necklaces}, together with our two
running examples of necklace sets:
$\cQ_{\mathrm{all}}=\{\bc,\bl,\bt\}^*$, the set of all necklaces, and
$\cQ_\lambda=\{\bc\bc,\bt,\bl\}$, with only three allowed necklaces,
with respective vertex generating functions $Q_{\mathrm{all}}(v,w,u)=1/(1-(v+w+u))$ and
$Q_\lambda(v,w,u)=v^2+w+u$.

\begin{figure}[t]
  \begin{center}
    \begin{minipage}{.16\linewidth}
      $w=\raisebox{-1.3em}{\includegraphics[scale=.3]{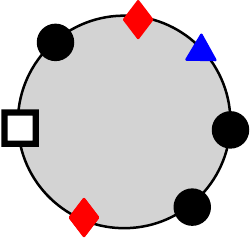}}$
    \end{minipage}
    \begin{minipage}{.8\linewidth}
      \begin{align*}
        \mathcal{Q}_{\mathrm{all}}&=\left\{\raisebox{-1.3em}{\includegraphics[scale=.3]{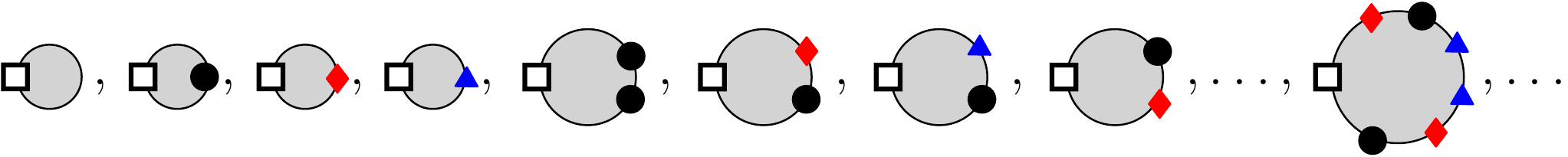}}\right\}\\
        \mathcal{Q}_\lambda&=\left\{\raisebox{-.7em}{\includegraphics[scale=.3]{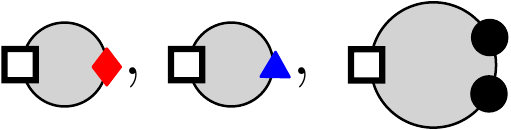}}\right\}
      \end{align*}
    \end{minipage}
  \end{center}
  \caption{\label{fig:Necklaces}The graphical representation of the necklace $w=\bc\bl\bt\bc\bc\bl$, and the necklace sets $\cQ_{\mathrm{all}}=\{\bc,\bl,\bt\}^*$ and $\cQ_\lambda=\{\bc\bc,\bt,\bl\}$.}
\end{figure}

As illustrated by Figure~\ref{fig:Q-trees}, a \emph{rooted $\cQ$-tree} is a $\ws$-rooted plane tree (\emph{aka} ordered tree) with black and red edges such that
\begin{itemize}
\item each vertex is  a copy of a necklace of $\cQ$,
\item each black edge connects a $\bc$-pearl to a $\ws$-pearl, \emph{i.e.}, takes the form $\be$,
\item each red edge connects a $\bl$-pearl to a $\ws$-pearl, \emph{i.e.}, takes the form $\re$,
\item each pearl is incident to one edge except the
  $\ws$-root and the $\bt$-pearls which are \emph{free}, that is, incident to no edge.
\end{itemize}
Here a tree is a connected graph without cycle and it is $\ws$-rooted if a $\ws$-pearl is distinguished, while the qualifier \emph{plane} (or \emph{ordered}) refers to the fact that the necklace structure on vertices induces a linear order on the children of each vertex.
We are only interested
in finite trees, so we shall assume from
now on that $\mathcal{Q}$ contains at least one vertex with no child,
that is, at least one necklace without $\bc$- and $\bl$-pearls, and for the sake of simplicity we first concentrate on the case in which the size $|\tau|$ of a $\cQ$-tree $\tau$ is the number of its
vertices (but this restriction can easily be lifted to allow vertices of various sizes, as explained in Section~\ref{sec:generic}, at the price of slightly more cumbersome notations). By construction, the size is also the number $|\tau|_\ws$ of
$\ws$-pearls (since each vertex carries one $\ws$-pearl), and the number of edges plus one:
$|\tau|=|\tau|_{\ws}=|\tau|_\bc+|\tau|_\bl+1$.  

\begin{figure}[t]
  \begin{center}
      \includegraphics[scale=.3]{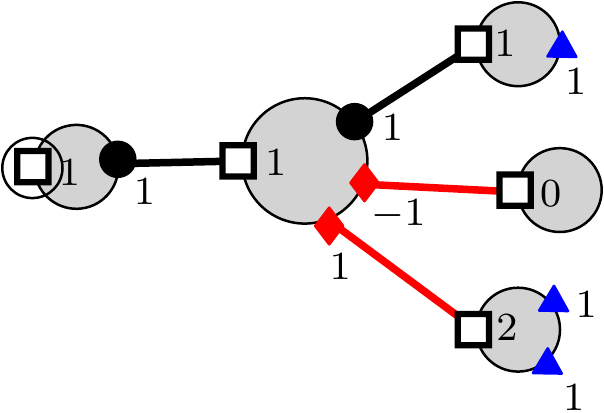}\quad
      \includegraphics[scale=.3]{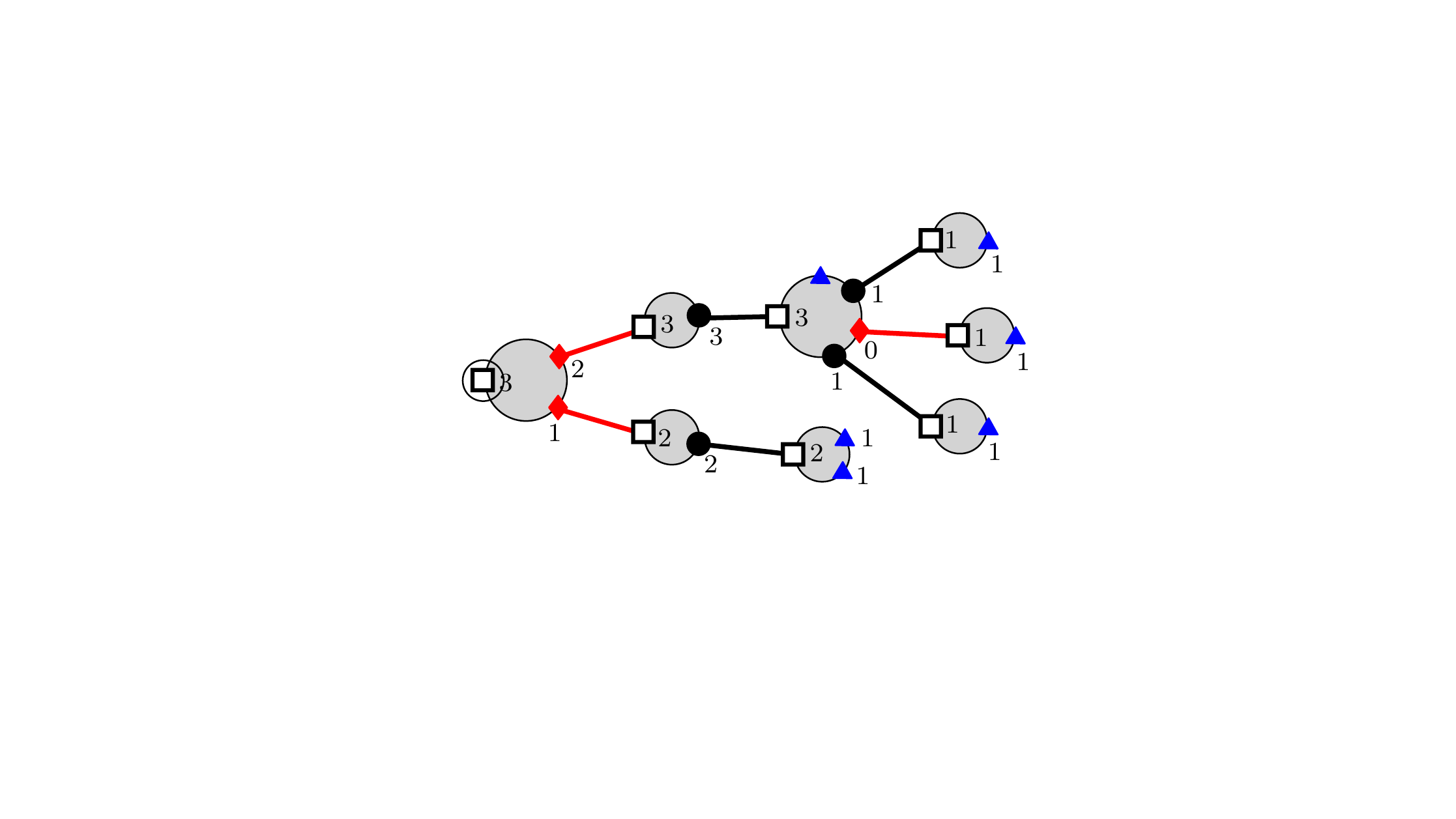}\quad
      \includegraphics[scale=.3]{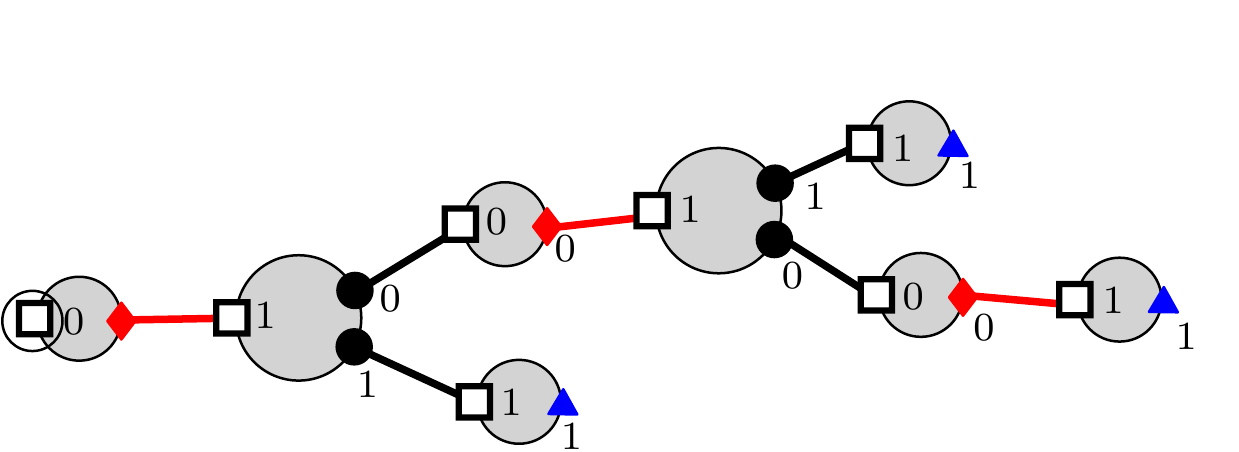}
  \end{center}
  \caption{\label{fig:Q-trees}Three rooted $\cQ_{\mathrm{all}}$-trees with their excess labels, and with root indicated by a $\fullmoon$ around the root $\ws$-pearl.  The third one is also a $\cQ_\lambda$-tree. The second and third ones  are non-negative while the first one is not. }
\end{figure}

The subtree $\tau_x$ of a rooted $\cQ$-tree $\tau$ at a pearl $x$
consists of $x$ and all the vertices, edges and pearls that are on the other side of
$x$ with respect to the root of $\tau$. The \emph{excess} of a pearl
$x$ in the $\cQ$-tree $\tau$ is the difference between the number
of $\bt$- and $\bl$-pearls in the subtree planted at $x$, $x$
included: $\mathrm{exc}(x)=|\tau_x|_\bt-|\tau_x|_\bl$. The excess of a
vertex $v$ is the excess of its local root (the only $\ws$-pearl on
$v$), and the excess of $\tau$ is the excess of its root, that
is $\mathrm{exc}(\tau)=|\tau|_\bt-|\tau|_\bl$.
A \emph{non-negative $\cQ$-tree} is a rooted $Q$-tree whose excess
is non-negative at each pearl.
These definitions are illustrated by Figure~\ref{fig:Q-trees}.
Observe that the non-negativity condition at each pearl in the
definition of non-negative $\cQ$-trees is in general more restrictive than
just saying that the excess is non-negative at each vertex: for instance this latter
condition would be satisfied by the leftmost tree in
Figure~\ref{fig:Q-trees} whereas it is not non-negative due to its $\bl$-pearl with excess $-1$.

The following alternative recursive characterization of non-negative
$\cQ$-trees follows immediately from the standard root vertex decomposition of ordered trees and its translation in terms of generating functions (cf \cite[Chapters I, III]{FS}):
\begin{pro}\label{pro:cata}
  Let $\cQ$ be as in Equation~\eqref{eq:Q} and $\mathcal{F}_{k}$ denote the
  set of non-negative $\cQ$-trees with excess $k$ and
  $\mathcal{F}=\bigcup_{k\geq0}\mathcal{F}_k$. Then the family
  $\mathcal{F}$ of non-negative $\cQ$-trees admits the following
  \emph{catalytic specification}:
  \begin{align}\label{eqn:cataspec}
  \mathcal{F}\equiv \cQ(\bc\!\!\rule[2pt]{10pt}{1pt}\mathcal{F},\bl\!\!\textcolor{red}{\rule[2pt]{10pt}{1pt}}\mathcal{F}^+,\bt),
  \end{align}
  meaning that each tree of $\mathcal{F}$ can be uniquely obtained from a necklace $s\in\cQ$ upon attaching
  \begin{itemize}
  \item a black edge carrying a tree of $\mathcal{F}$ to each $\bc$-pearl of $s$,
  \item and a red edge carrying a tree of $\mathcal{F}^+=\mathcal{F}\setminus\mathcal{F}_0$ to each $\bl$-pearl of $s$.
  \end{itemize}
  In particular, the gf $F(u)\equiv F(t,u)=\sum_{\tau\in\mathcal{F}}t^{|\tau|} u^{\mathrm{exc}(\tau)}$ is the unique formal power series solution of Equation~\eqref{eqn:eq1} with $Q$ given by Equation~\eqref{eq:Q}, and $F(u)-F(0)=F^+(u)=\sum_{\tau\in\mathcal{F}^+}t^{|\tau|} u^{\mathrm{exc}(\tau)}$.
\end{pro}

For completeness let us observe that non-negative $\cQ$-trees can also be viewed as a generalization of \emph{parking trees} as studied in \cite{Chen, contat24, contat23,contatcurien21,Pan}, upon considering each $\bt$-pearl as the entry place of a car that seeks a free $\bl$-pearl to park on on its path to the root: it can indeed be checked that a $\cQ$-tree is non-negative if and only if all its cars manage to park (see \emph{e.g.} \cite{henri2025} for a detailed discussion of this link).

\subsection{Plane maps and their spanning trees}\label{sec:map}

\begin{figure}[t]
  \begin{center}
      \includegraphics[scale=.2,page=1]{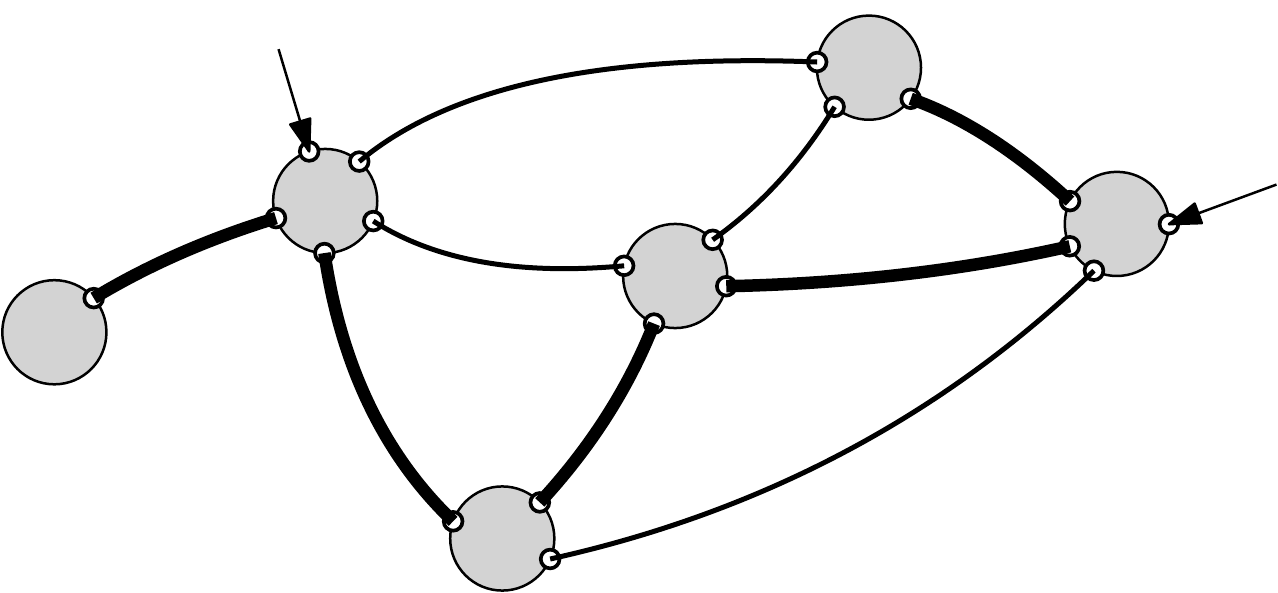}
      \includegraphics[scale=.2,page=2]{opening-closure-with-spanning.pdf}
      \includegraphics[scale=.2,page=3]{opening-closure-with-spanning.pdf}
  \end{center}
  \caption{\label{fig:generic-opening-spanning} (i) A tree-rooted blossoming map with 2 incoming half-edges and 9 edges, (ii) its clockwise orientation, (iii) and its clockwise opening, which is a blossoming tree with 11 incoming and 9 outgoing half-edges.}
\end{figure}

It will prove convenient to describe our bijections graphically, in
terms of non-crossing arc systems built around plane trees, or of
planar maps endowed with spanning trees.  These are standard
combinatorial concepts: the basic definitions and results we rely upon
can be found for instance in \cite[Chap.9] {appliedlothaire} or
\cite[Sec. V.2 and V.3]{chapterMaps}, but we recall here some standard definitions
and results that will be useful in the rest of the text.

A \emph{planar map} is a proper embedding of a connected graph in the
sphere, considered up to orientation-preserving homeomorphisms of the
sphere. A \emph{plane map} is a planar map with a distinguished face,
called the \emph{unbounded face}. Equivalently a plane map can
be viewed as a proper embedding of a connected graph in the plane
considered up to orientation-preserving homeomorphisms of the plane,
with the distinguished face being literally the unbounded one. In
this context, \emph{plane trees} can be described equivalently as
plane maps without cycles, plane maps with one face, or, plane maps
with one more vertices than edges. A \emph{tree-rooted map} is a pair
$(\mu,\tau)$ where $\mu$ is a plane map and $\tau$ is a subset of the
edges of $\mu$ that forms a spanning tree of its vertices, or
equivalently, such that between any two vertices of $\mu$ there is a
unique simple path using only edges of $\tau$. Finally we extend
slightly the definition of plane map to allow for dangling oriented
half-edges in the unbounded face: the resulting decorated maps are
called \emph{blossoming maps}, and in particular \emph{blossoming
trees} when they have only one face. An example of tree-rooted
blossoming map is given in
Figure~\ref{fig:generic-opening-spanning}(i).

The first ingredient we shall need is the \emph{clockwise
(resp. counterclockwise) opening} of a tree-rooted map $(\mu,\tau)$,
which is the blossoming tree obtained from $(\mu,\tau)$ along the following two steps:
\begin{itemize}
\item First orient the edges of $\mu$ that are not in
$\tau$ in clockwise (resp. counterclockwise) direction around $\tau$ \emph{i.e.}
each edge $e$ is oriented so that the unique simple cycle it forms
with edges of $\tau$ has the unbounded region of the plane map $\mu$
on its left (resp. right) hand-side. This operation is illustrated by Figure~\ref{fig:generic-opening-spanning} (ii).
\item Then break each (oriented) edge of $\mu\setminus\tau$ into an
  incoming half-edge and an outgoing half-edge. This operation is illustrated by Figure~\ref{fig:generic-opening-spanning} (iii).
\end{itemize}
As illustrated by Figure~\ref{fig:generic-opening}, the following
proposition is a reformulation in the language of maps of the well
known correspondence between arc diagrams and balanced parenthesis
strings, see \emph{e.g.} \cite[Prop. 6.1.1]{stanley} and \cite[Theorem V.6]{chapterMaps}.
\begin{pro}\label{pro:closure}
  Let $p$, $q$ and $\ell$ be non-negative integers. Then clockwise (resp.
  counterclockwise) opening is a vertex degree preserving bijection
  between
  \begin{itemize}
  \item tree-rooted blossoming maps with $p+1$ vertices, and $p+q$
    edges, and $\ell$ incoming half-edges in the unbounded face (but no
    outgoing half-edges),
  \item and blossoming trees with $p$ edges, and $q$ outgoing and $q+\ell$ incoming half-edges in the unbounded face.
  \end{itemize}
\end{pro}
\begin{figure}[t]
  \begin{center}
      \includegraphics[scale=.17,page=1]{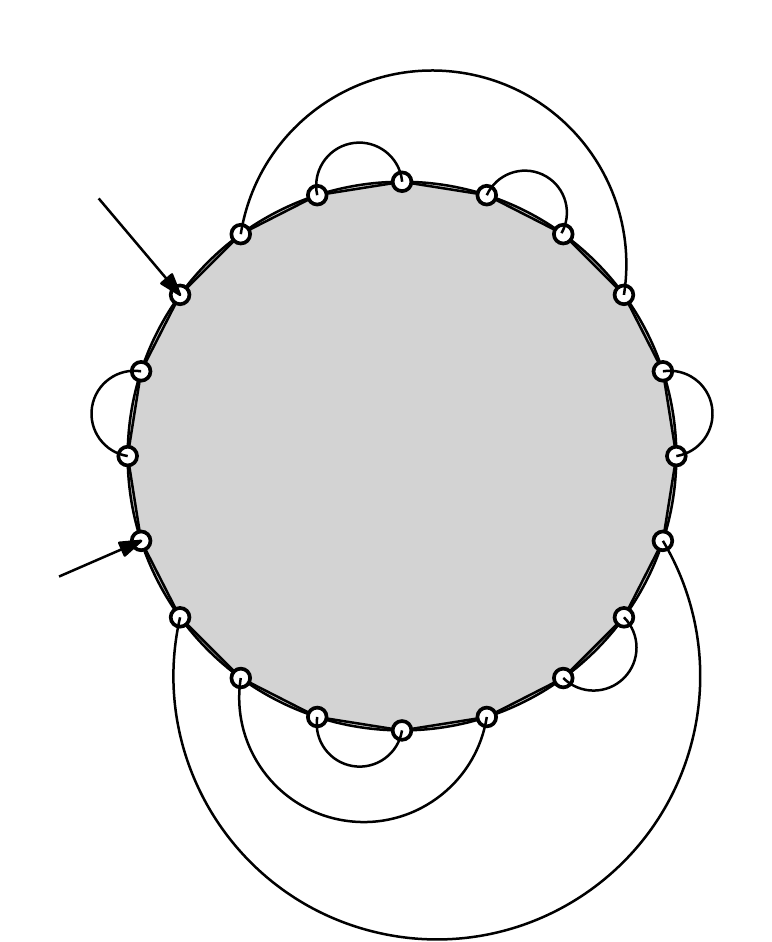}
      \includegraphics[scale=.17,page=2]{opening-closure.pdf}
      \includegraphics[scale=.17,page=3]{opening-closure.pdf}
      \includegraphics[scale=.17,page=4]{opening-closure.pdf}
      \includegraphics[scale=.17,page=5]{opening-closure.pdf}
      \includegraphics[scale=.17,page=6]{opening-closure.pdf}
  \end{center}
  \caption{\label{fig:generic-opening} (i) A one-vertex tree-rooted blossoming map with 2 incoming half-edges and 9 edges (here the spanning tree is reduced to the single vertex, which has been magnified) (ii) its clockwise orientation, (iii) its clockwise opening, a one-vertex blossoming tree with 11 incoming and 9 outgoing half-edges, (iv)-(vi) three intermediary steps in the iterative clockwise closure.}
\end{figure}

The inverse of clockwise (resp. counterclockwise) opening is called
the \emph{clockwise (resp. counterclockwise) closure}.  A standard way
to perform the closure, as illustrated by Figure~\ref{fig:generic-opening}, is by iteration of local matchings: a
\emph{clockwise matching pair} of a blossoming map $\mu$ is a pair
$(x,y)$ where $x$ is an outgoing half-edge and $y$ is an incoming
half-edge such that no other half-edge is met
while following the border of the unbounded face of $\mu$ in clockwise
direction from $x$ to $y$.
\begin{pro}\label{pro:iter}
  The clockwise closure of a blossoming map $\mu$ can be computed by
  iteratively matching clockwise matching pairs until no outgoing
  edges remain.
\end{pro}
Again, this is a mere reformulation of the standard algorithm to match parenthesis in a balanced parenthesis string, as discussed for instance in \cite[Prop. 6.1.1]{stanley}, \cite[Chapter 11]{lothaire}, or again \cite[Theorem V.6]{chapterMaps}.

\subsection{The closure and rewiring of a non-negative $\cQ$-tree} \label{sec:closure}\label{sec:rewiring}
 From now on we view $\mathcal{Q}$-trees as plane
maps with one face and decorated edges and vertices: necklaces are viewed
as vertices and pearls as colored endpoints of edges, and we consider
more generally plane maps with such decorated edges and vertices. In
particular, a plane map is \emph{rooted} if one of its pearl is
distinguished as \emph{the root pearl}.  Observe that the root pearl is in
general not required to be incident to the unbounded face, although
this will often be an interesting case.

Around a non-negative $\cQ$-tree $\tau$, as illustrated by Figure~\ref{fig:Corners}, let
\begin{itemize}
\item a \emph{left $\bl$-corner} be the exterior angular sector following
  a red edge in counterclockwise direction around a $\bl$-pearl,
\item a \emph{$\bt$-corner} be the exterior angular sector around a $\bt$-pearl,
\end{itemize}
and observe that the fact that $\tau$ is non-negative implies that there are at least as many $\bt$-corner
as $\bl$-corners. Define then the ($\bl$-to-$\bt$ clockwise) \emph{closure} $c(\tau)$ of
$\tau$ as the plane map obtained as follows: grow an outgoing half edge in
each left $\bl$-corner of $\tau$, and an incoming half-edge in each $\bt$-corner;
then apply the clockwise closure to the resulting blossoming map to form
new $\bl$-to-$\bt$ clockwise edges, hereafter called \emph{blue} edges
($\pe$); finally remove the extra inserted incoming half-edges (if there were more $\bl$-corners than $\bt$-corners). Equivalently,
in view of Proposition~\ref{pro:iter}, the closure can be constructed by
iteratively matching unmatched left $\bl$-corners that are followed by
an unmatched $\bt$-corner in clockwise direction around the tree, to
form a planar system of non-crossing $\bl$-to-$\bt$ clockwise
edges. This construction, illustrated by Figures~\ref{fig:Corners}
and~\ref{fig:closure}.
\begin{figure}[t]
  \begin{center}
    \begin{minipage}{.5\linewidth}
      \includegraphics[scale=.3,page=4]{Q-all-tree-bit-larger}\quad
    \end{minipage}
    \begin{minipage}{.3\linewidth}
      \includegraphics[scale=.3]{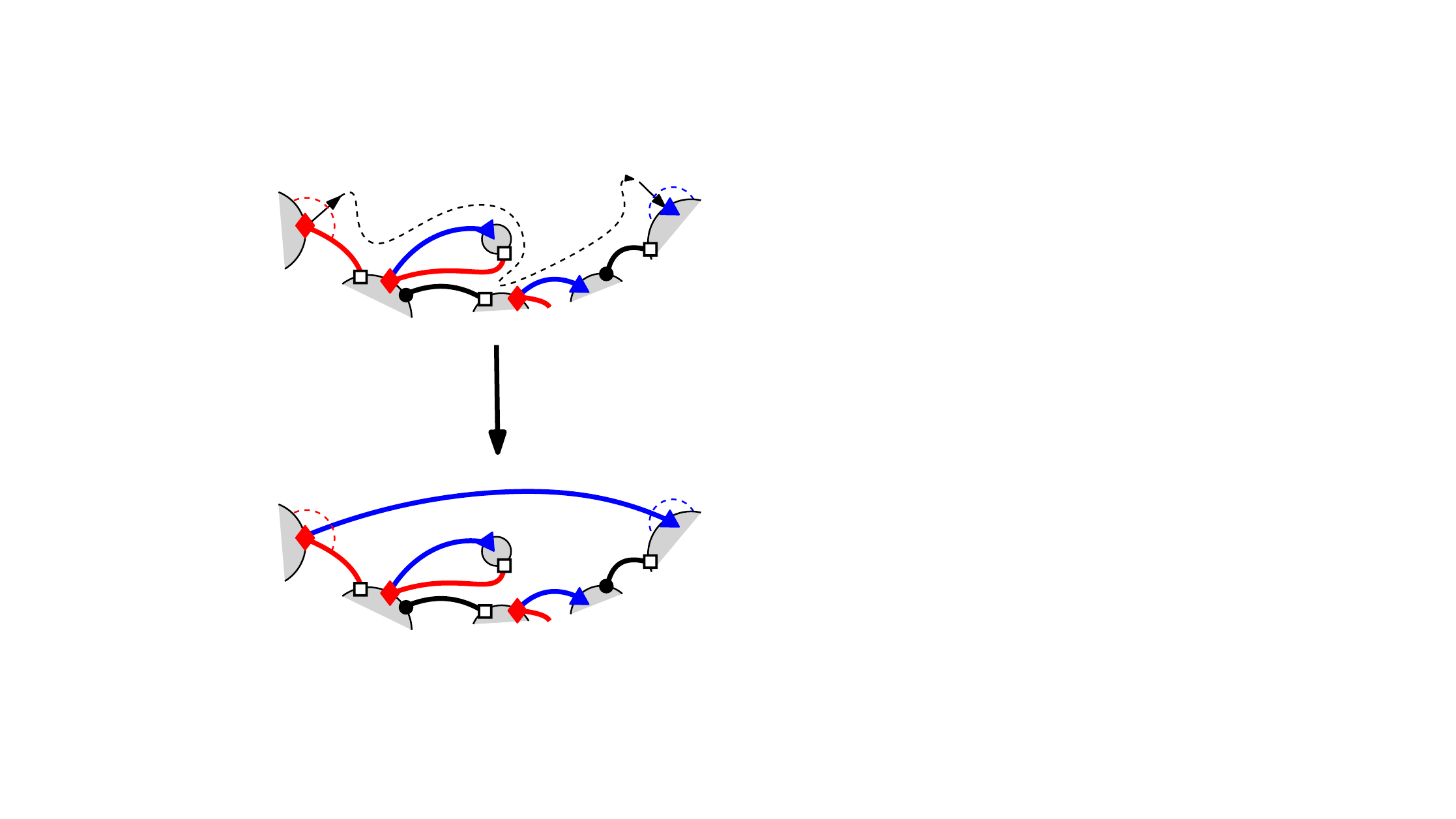}
    \end{minipage}
  \end{center}
  \caption{The left $\bl$-corners and $\bt$-corners of a non-negative $\cQ$-tree, and the
    matching of a left $\bl$-corner with the next available $\bt$-corner in
    clockwise direction.}
  \label{fig:Corners}
\end{figure}

\begin{figure}[t]
  \begin{center}
      \includegraphics[scale=.3,page=5]{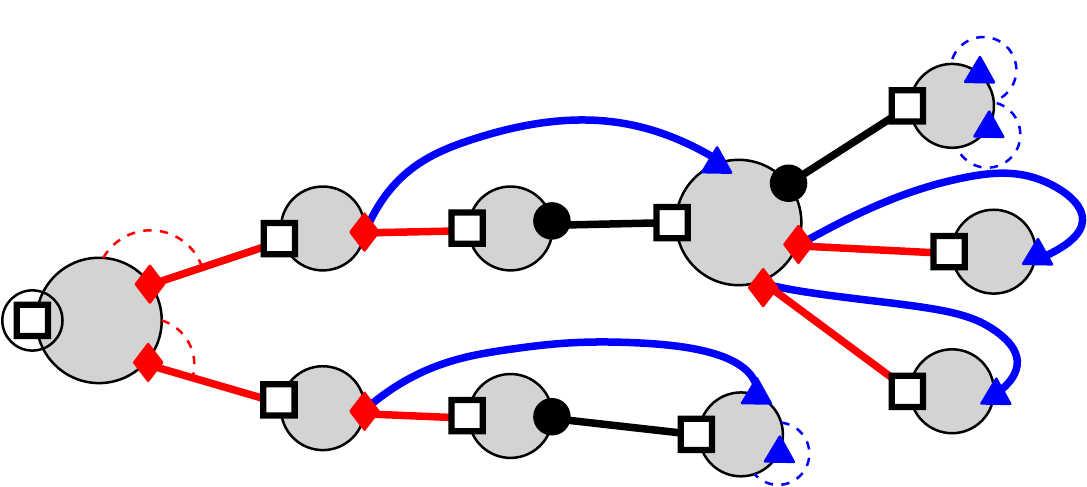}\quad
      \includegraphics[scale=.3,page=6]{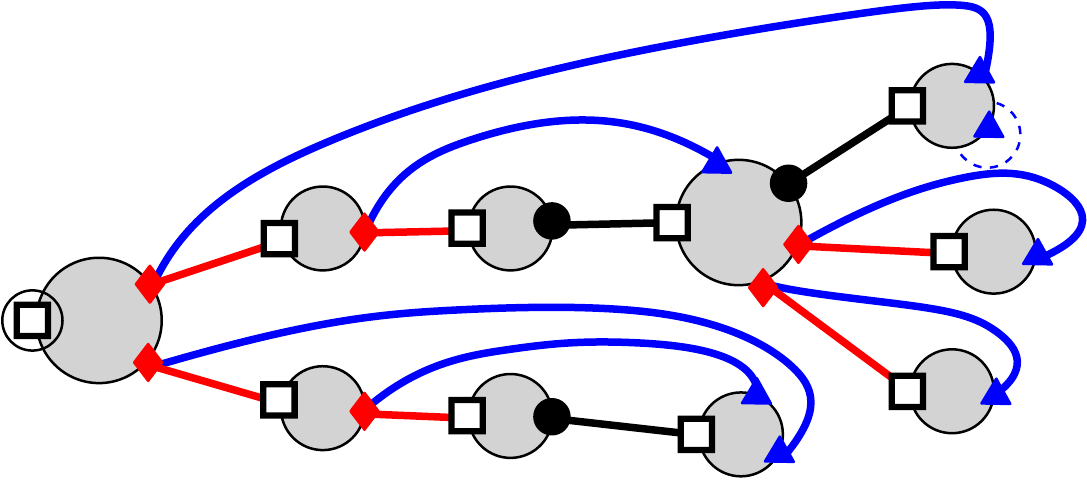}
  \end{center}
  \caption{The matchings after first iteration and final result of the
    $\bl$-to-$\bt$ clockwise closure of the tree of
    Figure~\ref{fig:Corners}.}
  \label{fig:closure}
\end{figure}

The \emph{rewiring} $\phi(\tau)$ of a $\cQ$-tree $\tau$ then consists in its closure followed by the removal of red edges, as illustrated by Figure~\ref{fig:rewiring}.

\begin{pro}\label{pro:rewiringtree}
  The rewiring $\phi(\tau)$ of a non-negative $\cQ$-tree $\tau$ is a tree with the same necklaces.
\end{pro}

In order to prove Proposition~\ref{pro:rewiringtree} we first gather some immediate consequences of the definition of the closure in the following lemma:
\begin{lem}\label{lem:lemma}
  The closure $c(\tau)$ of a non-negative $\cQ$-tree $\tau$ is a $\ws$-rooted plane map with vertices in $\cQ$, and black edges ($\be$), red edges ($\re$) and blue edges ($\pe$) such that:
  \begin{itemize}
  \item[(i)] The black and red edges of $\tau$ form a spanning tree of $c(\tau)$ and each blue edge $e$ is $\bl$-to-$\bt$ clockwise around $\tau$, or equivalently, the $\bl$-to-$\bt$ orientation of $e$ orients the unique simple cycle it forms with edges of $\tau$ in the clockwise direction. 
  \item[(ii)] The clockwise walk around each bounded face of $c(\tau)$ visits exactly one blue edge in $\bl$-to-$\bt$ direction and one red edge in $\ws$-to-$\bl$ direction,
    and these two edges share their $\bl$-pearl.
  \item[(iii)] Each $\bl$-pearl $x$ of $\tau$ is matched with a $\bt$-pearl in its subtree $\tau_x$.
  \item[(iv)] All unmatched $\bt$-pearls in $c(\tau)$ lie in its unbounded face.
  \end{itemize}
  The closure is injective and the inverse \emph{opening} construction consists in deleting blue edges.
\end{lem}
\begin{proof}
  Property (i) holds by definition of the $\bl$-to-$\bt$ clockwise closure. For Property (ii)
  observe first that each bounded face $F$ of $c(\tau)$ is created in
  the iterative process by the matching of a last blue edge $\pe$ that
  ``closes'' $F$, so that this edge will be visited in $\bl$-to-$\bt$
  direction by the clockwise walk around $F$, right after visiting the
  incident red edge in $\ws$-to-$\bl$ direction (see
  Figure~\ref{fig:Corners}); observe that no other red (resp. blue)
  edge can be enclosed in $\ws$-to-$\bl$ (resp. $\bl$-to-$\bt$)
  direction otherwise its $\bl$- (resp. $\bt$-) pearl would stand
  between the $\bl$- and $\bt$-endpoints of the closing edge, in
  contradiction with the matching rule that these two pearls have to follow one
  another to be matched.  Property (iii) is an immediate consequence
  of the non-negativity of $\tau$: for the matching edge of a
  $\bl$-pearl $x$ to turn around $\tau_x$ without finding any
  available $\bt$-pearl, one would need all the $\bt$-pearls in
  $\tau_x$ to have been matched by other $\bl$-pearls in $\tau_x$, a
  contradiction with the non-negativity at $x$. Similarly (iv) is an
  immediate consequence of the definition of the local closure, as a
  blue edge cannot enclose an unmatched $\bt$-pearl.

  Since the closure only add blue edges and since all blue edges in
  $c(\tau)$ arise from the closure it is clear that the opening,
  \emph{i.e.} the deletion of blue edges, is the inverse of the
  closure, and that the closure is injective.
\end{proof}

\begin{proof}[Proof of Proposition~\ref{pro:rewiringtree}]
  Observe that there are
  no closure edges between vertices of different subtrees of the root
  vertex of $\tau$ by Property (iii) above. As a consequence, if the
  root vertex of $\tau$ has type $w=x_1\ldots x_k$ with the pearl
  $x_i$ carrying a subtree $\tau_i$ (empty iff $x_i=\bt$), then the
  closure of $\tau$ is the concatenation at a root necklace of type
  $w$ of the rewirings $\phi(x_i-\tau_i)$, where $x_i-\tau_i$ denote the
  tree with root necklace of type $x_i$ with unique subtree
  $\tau_i$.

  The proof then goes by recurrence on the size. If there are $k\geq2$
  non trivial subtrees at the root vertex of $\tau$, then each $x_i-\tau_i$ is a
  smaller tree than $\tau$ and all $\phi(x_i-\tau_i)$ are trees according
  to the recurrence hypothesis. Hence so is $\phi(\tau)$. If instead
  there is only one $\bc$-planted subtree at the root vertex of $\tau$,
  \emph{i.e.} if $\tau=$\raisebox{.5pt}{$\scriptscriptstyle\square$}$\!\!{\fullmoon}\!\!\bc\!\!-\tau_1$, then $\tau_1$ is smaller than $\tau$ and by recursion hypothesis
  $\phi(\tau_1)$ is a tree, so that $\phi(\tau)=$\raisebox{.5pt}{$\scriptscriptstyle\square$}$\!\!{\fullmoon}\!\!\bc\!\!-\phi(\tau_1)$
  is a tree as well.  Finally if $\tau=$\raisebox{.5pt}{$\scriptscriptstyle\square$}$\!\!{\fullmoon}\!\!\bl\!\!\textcolor{red}{-}\tau_1$ then $\phi(\tau)$
  consists of a root node with type $\bl$ with a blue edge attached to
  the first clockwise available $\bt$-pearl on $\phi(\tau_1)$. Since
  $\phi(\tau_1)$ is a tree by recursion hypothesis, so is $\phi(\tau)$.
\end{proof}

\begin{figure}[t]
  \begin{center}
    \begin{minipage}{.44\linewidth}
      \centerline{\includegraphics[scale=.3,page=8]{Q-all-tree-bit-larger}}
    \end{minipage}
    \begin{minipage}{.1\linewidth}
      \[
      \begin{array}{c}
        \phi\\
        \longrightarrow\\
        \longleftarrow\\
        \bar\phi
        \end{array}
      \]
    \end{minipage}
    \begin{minipage}{.44\linewidth}
      \centerline{\includegraphics[scale=.3,page=1]{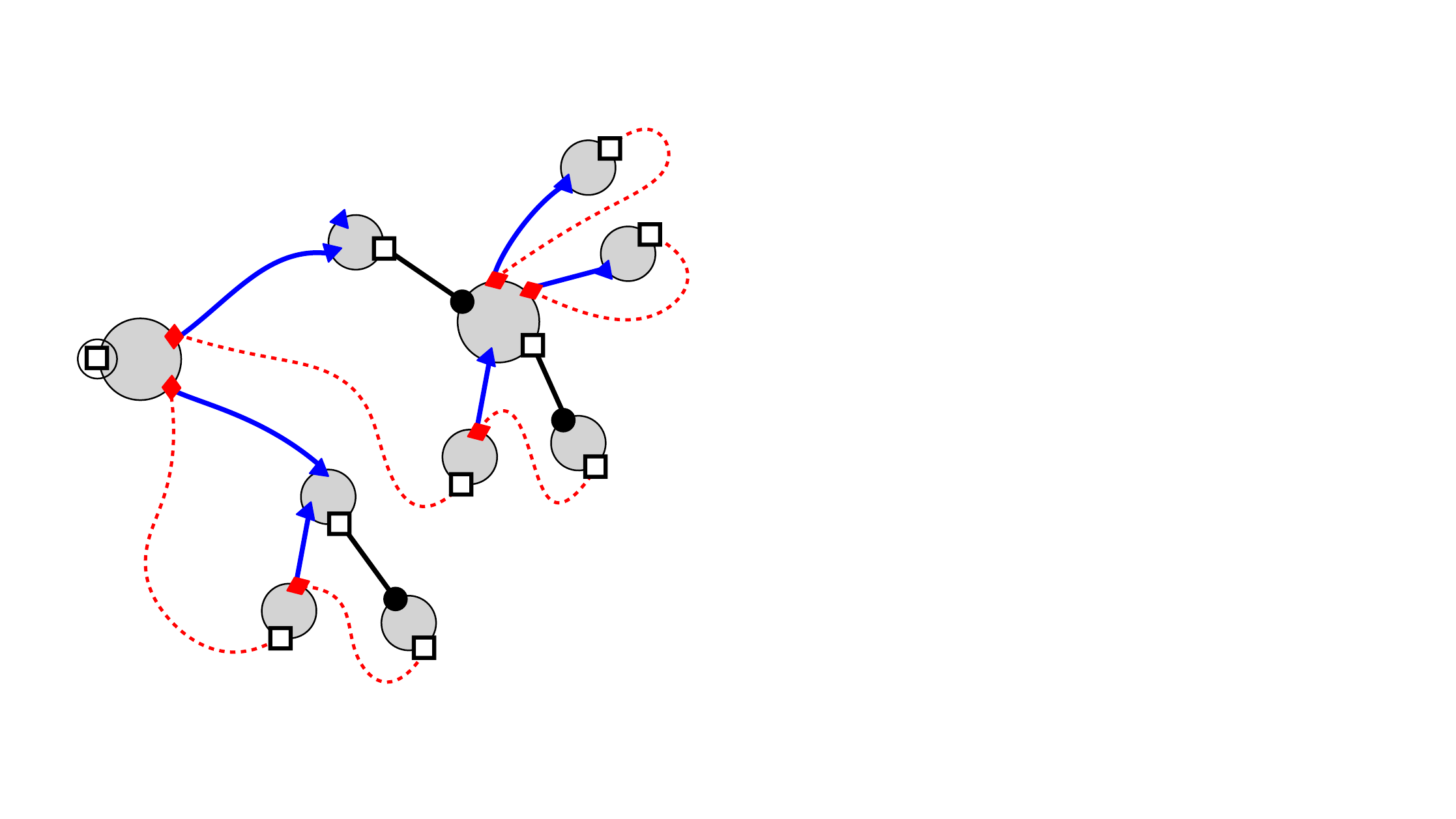}}
    \end{minipage}
  \end{center}
  \caption{The tree $\tau$ of Figure~\ref{fig:Corners} (red and black edges) with its closure edges (dashed blue lines), and its rewiring,  $\phi(\tau)$ (blue and black edges), with the inverse closure edges (dashed red lines). Observe that as plane maps, the two only differ by the dashing of blue versus red edges.}
  \label{fig:rewiring}
\end{figure}

An immediate consequence of Proposition~\ref{pro:rewiringtree} is that for any non-negative $\cQ$-tree $\tau$, the black and blue edges of $c(\tau)$ form a spanning tree and the rewiring $\phi(\tau)$ is the counterclockwise opening of $c(\tau)$ with respect to this spanning tree.

Observe that in the rewiring process some parts of the trees are
rerooted and get ``rotated'' with respect to the root, but the local necklace
structure is preserved: this motivates our choice to represent vertices
as necklaces.

\section{$\cQ$-companion trees and their decompositions}\label{sec:Qcompanion}
\subsection{$\cQ$-companion trees and the main bijection}\label{sec:defQcompanion}

By construction, rewiring replaces each red edge of the form $\re$ by
a blue edge of the form $\pe$ originating from the same $\bl$-pearl,
as illustrated by Figure~\ref{fig:rewiring}. This motivates the
following definition of \emph{rooted $\cQ$-companion trees} as pearl
rooted plane trees with black and blue edges such that
\begin{itemize}
\item each vertex is a copy of a necklace of $Q$,
\item each black edge connects a $\bc$-pearl to a $\ws$-pearl, \emph{i.e.}, takes the form $\be$,
\item each blue edge connects a $\bl$-pearl to a $\bt$-pearl, \emph{i.e.}, takes the form $\pe$,
\item each non-root $\bc$- or $\bl$-pearl is incident to exactly one edge, the
  root pearl is \emph{free} (\emph{i.e.}, incident to no edge) and each $\ws$- or $\bt$-pearl is incident to at most
  one edge, non-root free $\bt$-pearls being referred to as \emph{defects}.
\end{itemize}
As opposed to non-negative $\cQ$-trees which are always $\ws$-rooted,
we will also consider $\bc$-, $\bl$- and $\bt$-rooted $\cQ$-companion
trees. Moreover, we define an \emph{unrooted $\cQ$-companion tree} to be
an equivalence class of $\ws$-rooted $\cQ$-companion trees up to
rerooting: in other terms it is an unrooted plane tree satisfying the
conditions above and without free $\bc$- or $\bl$-pearl (like any
$\ws$-rooted $\cQ$-companion tree, and as opposed to $\bc$- or
$\bl$-rooted $\cQ$-companion trees, whose root is free).

On the one hand, according to Proposition~\ref{pro:rewiringtree}, the
image $\phi(\tau)$ of a non-negative $\cQ$-tree $\tau$ is a tree, on
the other hand it clearly satisfies all the constraints on vertices,
edges and pearls in the definition of $\cQ$-companion trees, so that
it is in fact a $\ws$-rooted $\cQ$-companion tree. However as we shall
see, not all $\ws$-rooted $\cQ$-companion trees can be obtained from
non-negative $\cQ$-trees by rewiring, and we have to study
the inverse construction to characterize the image of $\phi$.

By construction in a $\ws$- or $\bt$-rooted $\cQ$-companion tree
$\tau'$, $|\tau'|_\ws=|\tau'|_\bc+|\tau'|_\bl+1$, and in a $\bc$- or
$\bl$-rooted $\cQ$-companion tree $\tau'$,
$|\tau'|_\ws=|\tau'|_\bc+|\tau'|_\bl$. In particular, this implies
that the number of $\ws$-pearls that are free is equal to the number
of $\bl$-pearls plus one in a $\ws$-, $\bt$-rooted or unrooted $\cQ$-companion
tree, and to the number of $\bl$-pearls in a $\bc$- or $\bl$-rooted
one. This allows us to define, like in Section~\ref{sec:closure}, and as
illustrated by Figures~\ref{fig:rewiring}
and~\ref{fig:internaldefects}, the \emph{inverse ($\bl$-to-$\ws$
counterclockwise) closure} $\bar c(\tau')$ of a 
rooted or unrooted $\cQ$-companion tree $\tau'$ as the plane
map obtained by growing an outgoing half-edge in each right
$\bl$-corner and an incoming half-edge in each $\ws$-corner, 
applying counterclockwise closure, and removing the remaining incoming
half-edges. Equivalently, according to Proposition~\ref{pro:iter},
$\bar c(\tau')$ can be obtained upon matching iteratively right
$\bl$-corners that are followed by an unmatched $\ws$-corner in
counterclockwise direction around the tree to form a planar system of
non-crossing $\bl$-to-$\ws$ counterclockwise edges (\emph{i.e.} red
edges). If $\tau'$ is rooted then so is $\bar c(\tau')$, keeping the
same root pearl.
\begin{figure}[t]
  \begin{center}
    \includegraphics[scale=.3,page=2]{Q-Com-closure-details}\;
      \includegraphics[scale=.3,page=1]{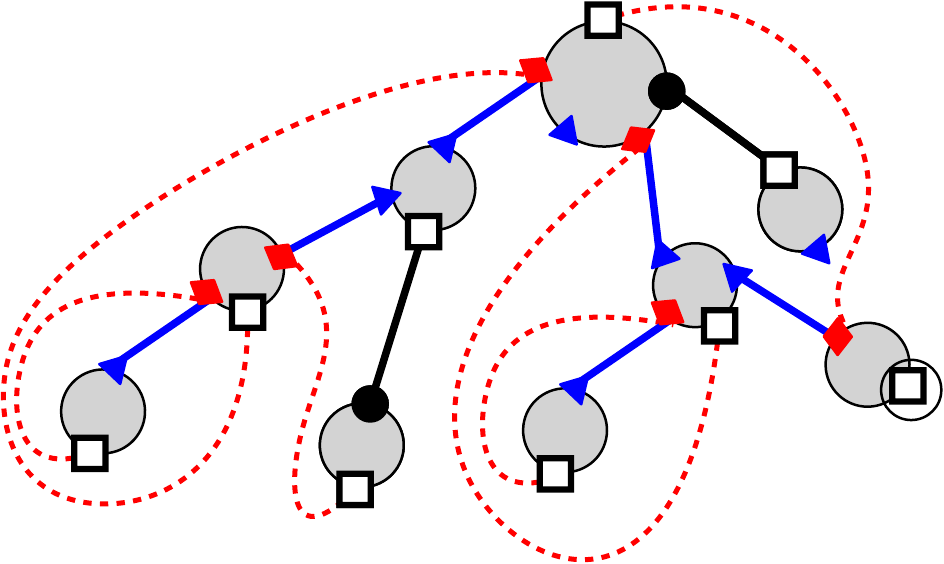}\;
      \includegraphics[scale=.3,page=3]{Q-Com-closure-details}
  \end{center}
  \caption{Three rooted $\cQ$-companion trees with their inverse
    closure edges (dashed red lines), and sharing the same underlying
    unrooted tree (Recall that the root pearl is indicated by
    $\fullmoon$, and that it is not necessarily represented as the top
    vertex). The leftmost tree is $\ws$-rooted and unbalanced, the
    middle one is $\ws$-rooted and balanced, they both have one
    internal and one external defects. The rightmost tree is
    $\bt$-rooted and unbalanced (observe that the root is not
    requested to be incident to the outer-face). It also has no
    internal defects as root $\bt$-pearls do not count among
    defects.}  \label{fig:internaldefects}
\end{figure}

If $\tau'$ is unrooted, or $\ws$- or $\bt$-rooted, then as observed
above its number of free $\ws$-pearls is one more than its number of
$\bl$-pearls, and the plane map $\bar c(\tau')$ has exactly one
unmatched $\ws$-pearl.  A $\ws$-rooted $\cQ$-companion tree $\tau'$ is
\emph{balanced} if it is rooted on this unique $\ws$-pearl that remains
unmatched in its closure $\bar c(\tau')$, \emph{unbalanced} otherwise. A
$\bc$- or $\bt$-rooted $\cQ$-companion tree is \emph{balanced} if its
root pearl remains incident to the unbounded face in its closure,
\emph{unbalanced} otherwise. Finally, $\bl$-rooted
$\cQ$-companion trees are declared \emph{unbalanced} since their root pearl is always matched during the closure.  These definitions
are illustrated by Figure~\ref{fig:internaldefects}.

The \emph{inverse rewiring} $\bar\phi(\tau')$ of a balanced
$\ws$-rooted $\cQ$-companion tree $\tau'$ is obtained from $\bar
c(\tau')$ by removing the blue edges. Recall that the defects of a
$\cQ$-companion tree are its non-root free $\bt$-pearls. A defect in a
$\ws$-rooted $\cQ$-companion tree $\tau'$ is said to be
\emph{external} (resp. \emph{internal}) if it lies in the unbounded
face (resp. in an inner face) of the inverse closure $\bar c(\tau')$
of $\tau'$.

\medskip
Our main combinatorial result is the following theorem:
\begin{thm}\label{thm:main}
  Rewiring and inverse rewiring are necklace-preserving bijections between
  \begin{itemize}
  \item the class $\mathcal{F}_k$ non-negative $\cQ$-trees with excess $k\geq 0$,
  \item and the class $\mathbf{B}_k$ of balanced $\ws$-rooted $\cQ$-companion trees with $k$ external defects and no internal defects.
  \end{itemize}
\end{thm}

This result immediately follows from Propositions~\ref{pro:rewiring}
and~\ref{pro:inverserewiring} below. Again we start by giving an
equivalent of Lemma~\ref{lem:lemma} that gathers some immediate
properties of the closure $\bar c$.

\begin{lem}\label{lem:lemma2} The closure $\bar c(\tau')$ of a $\cQ$-companion tree $\tau'$ is a planar map such that:
  \begin{itemize}
  \item[(i)] the black and blue edges of $\tau'$ form a spanning tree of $\bar c(\tau')$ and each red edge is  $\bl$-to-$\ws$ counterclockwise around $\tau'$,
  \item[(ii)] the bounded faces of $\bar c(\tau')$ satisfy the same property as in Lemma~\ref{lem:lemma}(ii),
  \item[(iii)] if there is a non-negative $\cQ$ tree $\tau$ such that $\tau'=\phi(\tau)$ then $\bar c(\tau')=c(\tau)$,
  \end{itemize}
\end{lem}
\begin{proof}
  Lemma~\ref{lem:lemma2}(i) follows from Proposition~\ref{pro:closure},  and Lemma~\ref{lem:lemma2}(ii) has the same proof as Lemma~\ref{lem:lemma}(ii). Lemma~\ref{lem:lemma2}(iii) follows from the observation that $\phi$ is the composition of a clockwise closure and a counterclockwise opening whose inverse is $\bar c$.
\end{proof}

\begin{pro}\label{pro:rewiring}
The rewiring of a non-negative $\cQ$-tree $\tau$ with excess $k$ is a
balanced $\ws$-rooted $\cQ$-companion tree without internal defects with $k$ external defects,
and $\bar \phi(\phi(\tau))=\tau$.
\end{pro}
\begin{proof}
In view of Proposition~\ref{pro:rewiringtree} and the definition of
$\cQ$-companion trees, it is immediate that the rewiring of a
non-negative $\cQ$-tree $\tau$ with excess $k$ is a $\ws$-rooted
$\cQ$-companion tree $\tau'$ with $k$ defects.  Moreover the deletion
of the blue edges identifies with the opening associated with the
inverse closure $\bar c$ (Proposition~\ref{pro:closure}): in other
terms, inverse closure allows us to recover the blue edges that are
deleted in rewiring, that is $c(\tau)=\bar c(\tau')$. Hence the final
property, as well as the fact that $\tau'$ is balanced since the root
of $\tau'$ is the only free $\ws$-pearl in $c(\tau)$, and in view of
Lemma~\ref{lem:lemma}(iv) $\tau'$ has no internal defects.
\end{proof}

\begin{pro}\label{pro:inverserewiring}
The inverse rewiring $\bar \phi(\tau')$ of a balanced $\ws$-rooted $\cQ$-companion tree $\tau'$ with $k$ external defects and no internal defects is a non-negative $\cQ$-tree with excess $k$, and $\phi(\bar\phi(\tau'))=\tau'$.
\end{pro}
\begin{proof}
  Let us first classify the possible subtrees $x-\tau''$ at a pearl $x$ of the root vertex $s$ of a balanced $\ws$-rooted $\cQ$-companion tree $\tau'$:
  \begin{itemize}
  \item If $x=\bc$ then $\tau''$ is a $\ws$-rooted $Q$-companion tree. 
  \item If $x=\bl$ then $\tau''$ is a $\bt$-rooted $Q$-companion tree. 
  \item If $x=\bt$ then $\tau''$ is a $\bl$-rooted $Q$-companion tree. 
  \end{itemize}
  Consider now the subtrees at $s$ in clockwise order starting after
  the root $\ws$-pearl, and let $x-\tau''$ be the first one such that
  $\tau''$ is unbalanced: in the closure of $\tau'$ the unbalance of
  $\tau''$ would cause $\tau'$ to be unbalanced, leading to a
  contradiction (in particular, when $\tau''$ is $\bl$-rooted, the root
  pearl of $\tau''$ would cause the imbalance of $\tau'$).  Hence all
  subtrees at $s$ are balanced:
  \begin{itemize}
  \item There is no $\bt$-pearl on $s$.
  \item Each $\bc$-pearl on $s$ carries a
    balanced $\ws$-rooted subtree $\tau''$.
  \item Each  $\bl$-pearl $x$ on $s$ matches by $\bar c$ a
  $\ws$-pearl $x'$ in the corresponding subtree $\tau''$ such that
  $\tau''_{x'}$, the tree obtained from $\tau''$ by rerooting at $x'$
  is a balanced $\ws$-rooted $Q$-companion tree. 
  \end{itemize}
  The proof that $\bar\phi(\tau')$ is a tree and that $\phi(\bar\phi(\tau'))=\tau'$ is then again by
  recurrence on the size.  If $\tau'$ is reduced to a root vertex
  without edges, or if there are at least two non empty subtrees, the
  result is immediate. It remains only to deal with the case where there is one non empty subtree:
  \begin{itemize}
  \item In the $\bc$-pearl case,
  $\bar\phi($\raisebox{.5pt}{$\scriptscriptstyle\square$}$\!\!{\fullmoon}\!\!\bc\!\!-\tau'')=$
  \raisebox{.5pt}{$\scriptscriptstyle\square$}$\!\!{\fullmoon}\!\!\bc\!\!-\bar\phi(\tau'')$
  where $\tau''$ is a balanced $\ws$-rooted $Q$-companion tree and
  $\bar\phi(\tau'')$ is by recurrence hypothesis a non-negative
  $Q$-trees so that $\bar\phi($\raisebox{.5pt}{$\scriptscriptstyle\square$}$\!\!{\fullmoon}\!\!\bc\!\!-\tau'')$ is.
  \item In the $\bl$-pearl case, the blue edge at $s$ is replaced via
    $\bar\phi$ by a red edge which reconnects $s$ to the root $x'$ of
    the non-negative $Q$-tree preimage $\bar\phi(\tau''_{x'})$ of the
    balanced $\ws$-rooted $Q$-companion tree $\tau''_{x'}$:
    $\bar\phi($\raisebox{.5pt}{$\scriptscriptstyle\square$}$\!\!{\fullmoon}\!\!\bc\!\!-\tau'')$
    is a non-negative $Q$-tree. This case is illustrated by the unique
    subtree of the root of the middle tree in
    Figure~\ref{fig:internaldefects}, or by both subtrees of the
    $\cQ$-companion tree of Figure~\ref{fig:rewiring}.
  \end{itemize}
  In each case the fact that the inverse construction is $\phi$ derives from the recurrence hypothesis using the same case analysis as above. 
\end{proof}

\subsection{Balanced, unrooted and rooted $\cQ$-companion trees without defects}\label{sec:unrooted}
Let us now relate bijectively the family $\mathbf{B}_0$ of balanced $\ws$-rooted
$\cQ$-companion trees without defects to the family $\mathbf C$ of
unrooted $\cQ$-companion trees without defects, and to the various
families $\mathbf{C}_\ws$, $\mathbf{C}_\bc$, $\mathbf{C}_\bl$ and
$\mathbf{C}_\bt$ of $\ws$-, $\bc$-, $\bl$- and $\bt$-rooted
$\cQ$-companion trees without defects (without the requirement of being
balanced).
\begin{figure}[t]
\centerline{    \includegraphics[scale=.3,page=7]{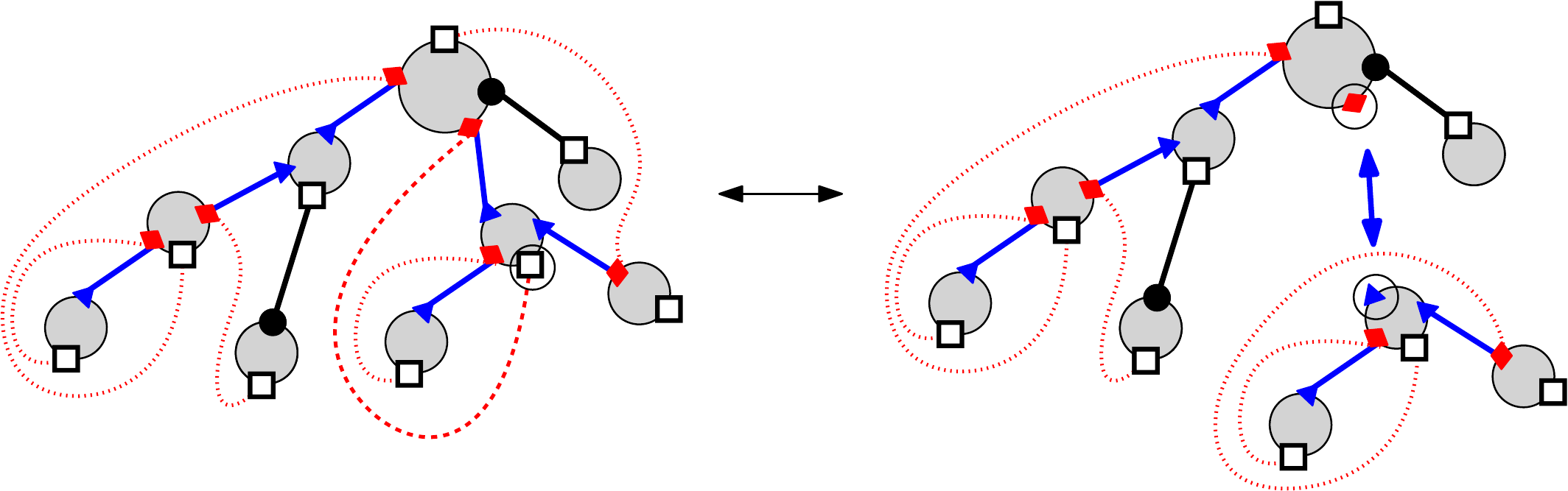}}
  \caption{An unbalanced $\cQ$-companion tree without defect and the corresponding pair of $\bl$- and $\bt$-rooted $\cQ$-companion trees without defects.\label{fig:unbalanced}}
\end{figure}
\begin{thm}\label{thm:un-balanced}
  There are necklace-preserving bijections between
  \begin{itemize}
  \item the class $\mathbf{B}_0$ of balanced $\ws$-rooted $\cQ$-companion trees without defects,
  \item and the class $\mathbf{C}$ of unrooted $\cQ$-companion trees without defects,
  \end{itemize}
  and between
  \begin{itemize}
  \item the class $\mathbf{U}_0$ of unbalanced $\ws$-rooted $\cQ$-companion trees without defects,
  \item and the class $\mathbf{C}_\bl\times\mathbf{C}_\bt$ of pairs
    made of a $\bl$-rooted $\cQ$-companion tree and a $\bt$-rooted
    $\cQ$-companion tree, both without defects,
  \end{itemize}
\end{thm}
\begin{proof}
  The first result is immediate: given a balanced $\cQ$-companion tree
  without defects, forgetting the root yields an unrooted
  $\cQ$-companion tree without defects with the same necklace
  distribution, and conversely given an unrooted $\cQ$-companion tree
  without defects, the inverse closure $\bar c$ allows us to identify the
  unique unmatched $\ws$-pearl where it can be rooted to have a
  balanced tree. See Figure~\ref{fig:internaldefects}.

  For the second result, illustrated by Figure~\ref{fig:unbalanced}, observe that a $\ws$-rooted $\cQ$-companion
  tree $\tau'$ is unbalanced if its root $\ws$-pearl $x$ is matched by
  a $\bl$-pearl $y$ during closure. The $\bl$-pearl $y$ carries in
  $\tau'$ a blue edge leading to a $\bt$-pearl $z$. Breaking this blue
  edge and rerooting the resulting two trees on the pearls $y$ and $z$
  respectively yield a $\bl$-rooted and a $\bt$-rooted $\cQ$-companion
  trees, again both without defects. Conversely given such a pair of
  rooted trees one can join the $\bl$-root pearl $y$ of the first tree
  to the $\bt$-root pearl $z$ of the second tree with a blue edge and
  root the resulting tree on the $\ws$-pearl that is matched with $y$
  in the closure $\bar c$. This construction is clearly always well
  defined and it is the inverse of the previous one.
\end{proof}


\subsection{The decomposition of rooted $\cQ$-companion trees without defects}\label{sec:decQcompanion}
\begin{figure}[t]
\centering
\includegraphics[scale=.4]{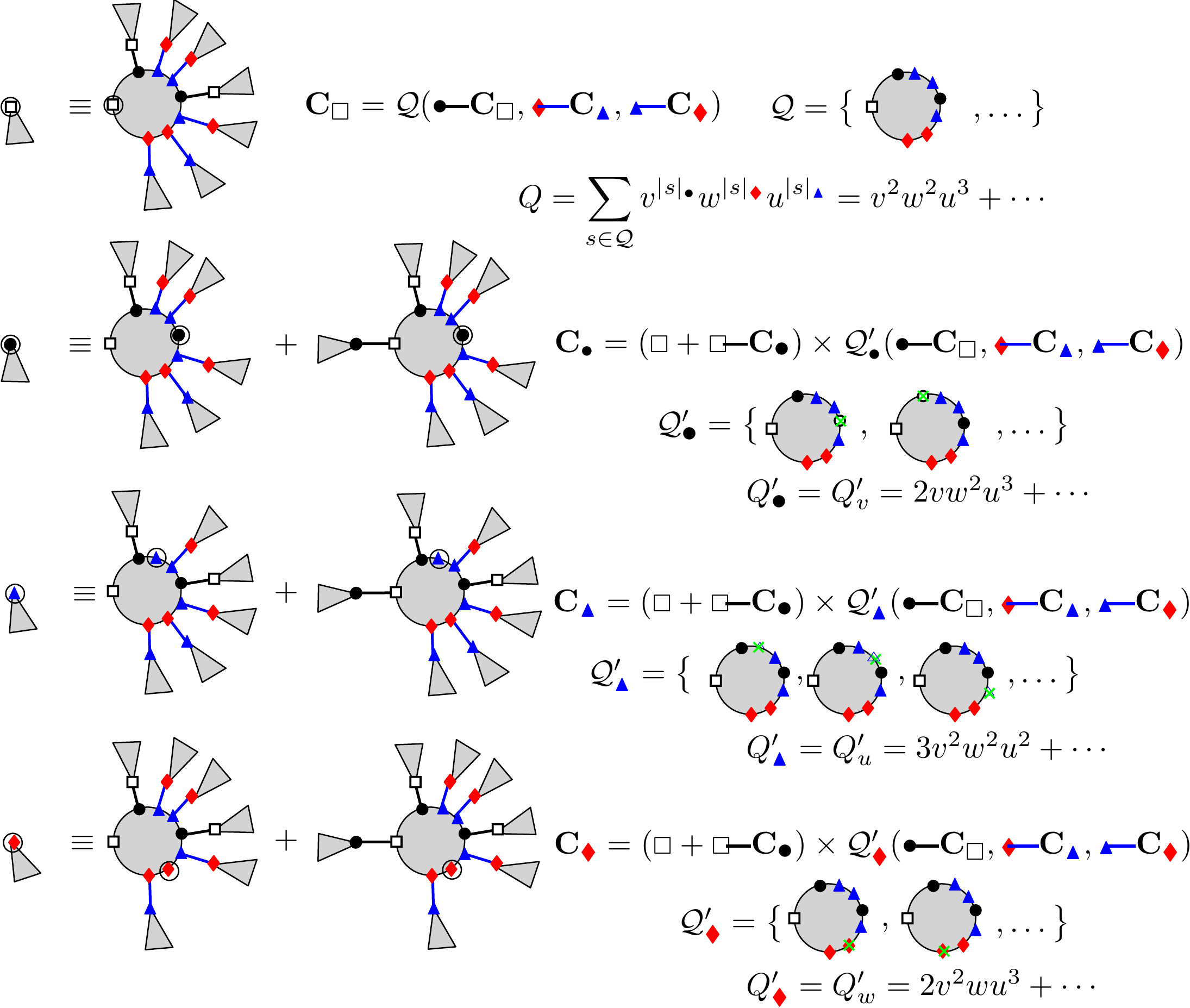}
  \caption{\label{fig:QNodes} The decomposition of $\mathbf{C}_\wc$, illustrated in the case of a vertex with necklace type $\bc\bt\bt\bc\bt\bl\bl$, and the derived necklaces involved in the decompositions of $\mathbf{C}_\bc$, $\mathbf{C}_\bt$ and $\mathbf{C}_\bl$.}
\end{figure}

In view of the previous section, it useful to consider more closely
the families $\mathbf{C}_\ws$, $\mathbf{C}_\bc$, $\mathbf{C}_\bl$ and
$\mathbf{C}_\bt$, and since these are families of rooted trees without
balance conditions, a natural approach is to perform a root vertex
decomposition.  The analysis of possible root necklaces in
$\cQ$-companion trees is illustrated by Figure~\ref{fig:QNodes}:
\begin{itemize}
\item the set of root vertex types of $\ws$-rooted $\cQ$-companion trees is $\cQ$, since each necklace $s\in\cQ$ has exactly one $\ws$-pearl,
\item the set $\cQ'_\bc$ of root vertex types of $\bc$-rooted $\cQ$-companion trees is the union for all necklaces $s\in\cQ$ of the $|s|_\bc$ different  rerootings of $s$ on a $\bc$-pearl,
\item the set $\cQ'_\bl$ of $\bl$-rooted $\cQ$-companion trees and $\cQ'_\bt$ of $\bt$-rooted $\cQ$-companion trees are obtained similarly.
\end{itemize}
Any $\ws$-rooted $\cQ$-companion tree without defects can thus be uniquely produced by selecting a necklace $s\in\cQ$ together with $|s|_\bc$ subtrees from $\cQ_\ws$, $|s|_\bl$ subtrees from $\cQ_\bt$ and $|s|_\bt$ subtrees from $\cQ_\bl$, and attaching these subtrees to the pearls of $s$. This operation is summarized as $\mathbf{C}_\ws\equiv{\cQ}(\bcedge\mathbf{C}_\ws,\bledge\mathbf{C}_\bt,\btedge\mathbf{C}_\bl)$.

The same approach allows us to deal with $\bc$-rooted $\cQ$-companion trees without defects, upon taking $s\in\cQ_\bullet'$ and adding a possibly empty extra subtree in $\cQ_\bullet$ to attach to the $\ws$-pearl of $s$. The other classes $\mathbf{C}_\bl$ and $\mathbf{C}_\bt$ admit similar decompositions. As a direct consequence, we have the following theorem.

\begin{thm}\label{pro:CFS}
  The standard root vertex decomposition of multi-type rooted trees yields the following context-free specification of rooted $\cQ$-companion trees without defects:
\begin{align}\label{eqn:combiCFS}
  \left\{
  \begin{array}{rclcl}
    \mathbf{C}_\ws
    &\equiv&{\cQ}(\bcedge\mathbf{C}_\ws,\bledge\mathbf{C}_\bt,\btedge\mathbf{C}_\bl),\\
    \mathbf{C}_\bc
    &\equiv&(\ws+\wsedge\mathbf{C}_\bc)\times {\cQ}'_\bc (\bcedge\mathbf{C}_\ws,\bledge\mathbf{C}_\bt,\btedge\mathbf{C}_\bl),\\
    \mathbf{C}_\bl
    &\equiv&(\ws+\wsedge\mathbf{C}_\bc)\times{\cQ}'_\bl (\bcedge\mathbf{C}_\ws,\bledge\mathbf{C}_\bt,\btedge\mathbf{C}_\bl),\\
    \mathbf{C}_\bt
    &\equiv&(\ws+\wsedge\mathbf{C}_\bc)\times{\cQ}'_\bt (\bcedge\mathbf{C}_\ws,\bledge\mathbf{C}_\bt,\btedge\mathbf{C}_\bl),\\
    \mathbf{C}^{\textcolor{black}{\fullmoon}}
    &\equiv&(\ws+\wsedge\mathbf{C}_\bc)\times{\cQ}(\bcedge\mathbf{C}_\ws,\bledge\mathbf{C}_\bt,\btedge\mathbf{C}_\bl),
  \end{array}
  \right.
\end{align}
where ${\cQ}(\bcedge\mathbf{C}_\ws,\bledge\mathbf{C}_\bt,\bcedge\mathbf{C}_\bl)$ denotes the set of trees obtained from a necklace of $\cQ$ by attaching to each $\bc$-pearl a subtree of the form $\bcedge\mathbf{C}_\ws$, to each $\bt$-pearl a subtree of the form $\btedge\mathbf{C}_\bl$, and to each $\bl$-pearl a subtree of the form $\bledge\mathbf{C}_\bt$, and similarly for the other equations, and where $\mathbf{C}^{\textcolor{black}{\fullmoon}}$ denotes the set of unrooted $\cQ$-companion trees without defects with a marked necklace. 

The first four equations yield our interpretation of System~\eqref{eq:CFS}, upon introducing the gf $C_\ws$, $C_\bc$, $C_\bl$ and $C_\bt$ of $\ws$-, $\bc$-, $\bl$- and $\bt$-rooted $\cQ$-companion trees, with $C_\ws=\sum_{\tau'\in\mathbf{C}_\ws}t^{|\tau'|}$, etc. 
\end{thm}
In the terminology of \cite{gs:BM-ICM,FS}, Theorem~\ref{pro:CFS} says
that rooted $\cQ$-companion trees without defects form a simple
variety of ordered (or plane) multi-type trees, and more precisely,
when the set of allowed necklaces $\cQ$ is finite, the resulting
specification is said to be \emph{$\mathbb{N}$-{algebraic}}.

\subsection{Non-negative $\cQ$-trees and rooted $\cQ$-companion trees}\label{sec:markedQtrees}

The first part of Theorem~\ref{thm:un-balanced} can be combined with Theorem~\ref{thm:main} to state the following:
\begin{cor}\label{cor:F0}
  There are necklace-preserving bijections
  \begin{align*}
    \mathbf{C}_\ws\equiv\mathbf{B}_0\cup (\mathbf{C}_\bl\times\mathbf{C}_\bt)
    \equiv\mathcal{F}_0\cup(\mathbf{C}_\bl\times\mathbf{C}_\bt).
  \end{align*}
  In particular, $f=F(t,0)$, the gf of non-negative
  $\cQ$-trees with excess 0,
  can be expressed in terms of the gf
  $C_{\ws}$, $C_\bl$, and $C_\bt$ of 
  $\ws$-rooted, 
  $\bl$-rooted, and $\bt$-rooted $\cQ$-companion trees as:
  \[
  f={C}_\ws- ({C}_\bl\cdot{C}_\bt),
  \]
  yielding our combinatorial interpretation of the first form of Equation~\eqref{eq:fCFS}.
\end{cor}

Next observe that Theorem~\ref{thm:main} immediately extends, upon marking an arbitrary necklace, to a necklace-preserving bijection between the classes $\mathcal{F}_0^{\fullmoon}$ and $\mathcal{C}^{\fullmoon}$. With the last equation of the system in Theorem~\ref{pro:CFS}, this yields the following:
\begin{cor}\label{cor:F0O}
  There are necklace-preserving bijections 
  \begin{align*}
  \mathcal{F}_0^{\fullmoon}\equiv\mathbf{C}^{\fullmoon}\equiv
  (\ws+\wsedge\mathbf{C}_\bc)\times{\cQ}(\bcedge\mathbf{C}_\ws,\bledge\mathbf{C}_\bt,\btedge\mathbf{C}_\bl),
   \end{align*}
   In particular the gf $\frac{\partial}{\partial t}f$ of non-negative $\cQ$-trees with excess 0 and with a marked necklace is given as
   \begin{align*}
  { \frac{d}{dt}}f&=(1+C_\bc)\cdot Q(C_\ws,C_\bt,C_\bl),
   \end{align*}
   and this yields our combinatorial interpretation of the second form of Equation~\eqref{eq:fCFS}.
\end{cor}

Finally observe that Theorem~\ref{thm:main} also extends to a necklace-preserving bijection between the family $\mathcal{F}_0^\bl$ of non-negative $\cQ$-trees with excess 0 and with a marked $\bl$-pearl, and the family 
$\mathbf{U}_0$ of unbalanced $\cQ$-companion trees without defects: indeed given a non-negative $\cQ$-tree $\tau$ with
excess 0 and a marked $\bl$-pearl $x$ of $\tau$, an unbalanced
$\cQ$-companion tree is obtained by rerooting $\phi(\tau)$ on the
$\ws$-pearl matched to $x$ in $\phi(\tau)$. This yields the following:
\begin{cor}\label{cor:F0pbl}
  There is a necklace-preserving bijection
  \[
  \mathcal{F}_0\cup\mathcal{F}_0^\bl\equiv \mathbf{B}_0\cup\mathbf{U_0}=\mathbf{C}_\ws.
  \]
  In particular, the gf $f^\bl$ of non-negative $\cQ$-trees with excess 0 and a marked $\bl$-pearl satisfies
  \[
  f+f^\bl=C_\ws.
  \]
\end{cor}
The latter corollary gives an interesting alternative interpretation of
the series $C_\ws$ as gf of possibly-$\bl$-marked non-negative $\cQ$-trees
with excess 0. Similar interpretations can be given to the series
$C_\bc$, $C_\bl$ and $C_\bt$.

\section{Examples, weighted extensions and linear special case}
\label{sec:appli}
\subsection{Planar $\lambda$-terms}

\begin{figure}[t]
  \begin{center}
    \includegraphics[scale=.3]{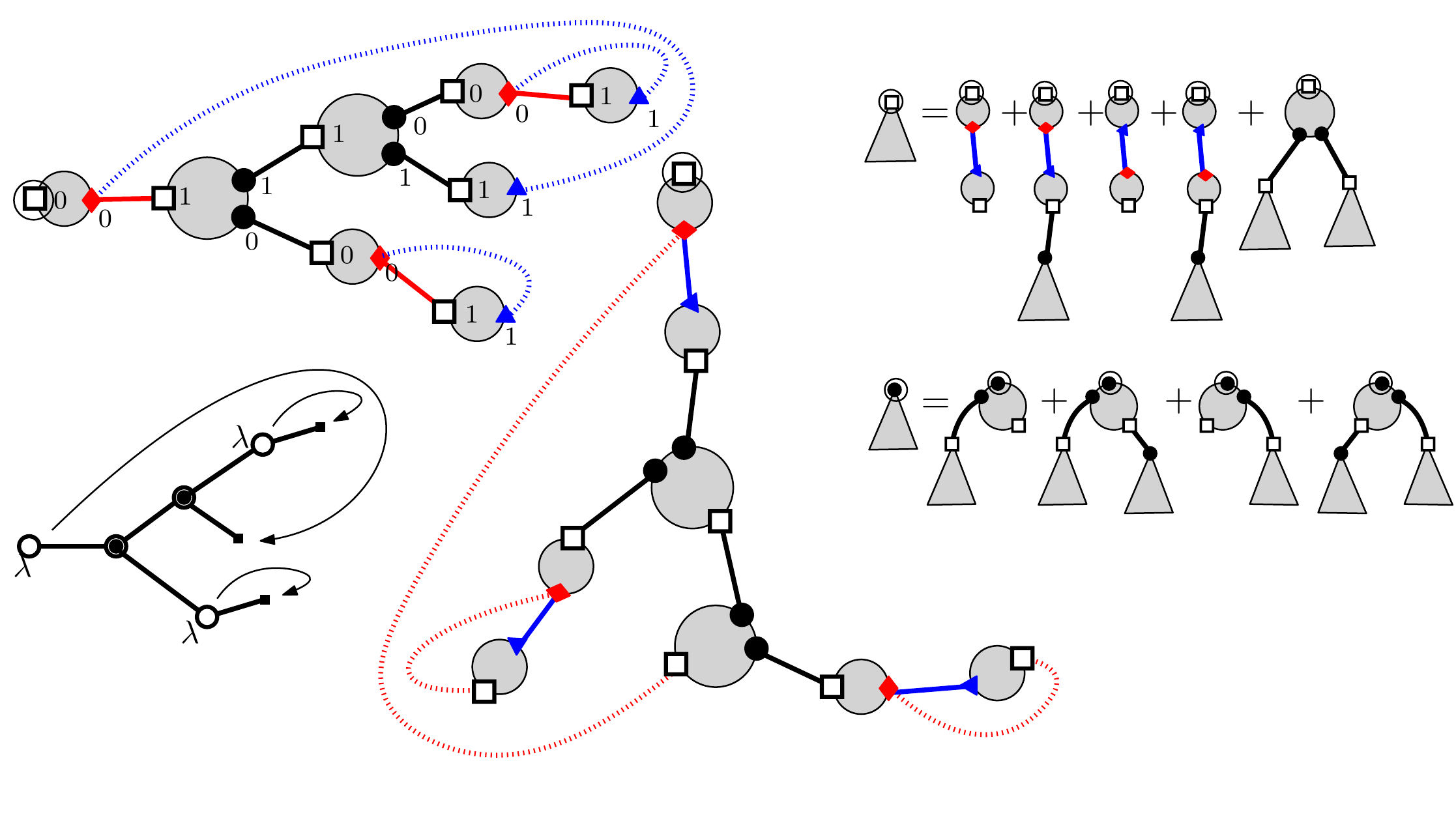}\vspace{-1em}
  \end{center}
  \caption{\label{fig:lambda-companion}A closed planar $\lambda$-term (bottom left), the corresponding non-negative $\cQ_\lambda$-tree (with blue closing edges) and its associated balanced $\cQ_\lambda$-companion tree (with red inverse closure edges). On the right the context-free specification  of $\ws$- and $\bc$-rooted $\cQ_\lambda$-companion trees without excess.}
\end{figure}

Planar $\lambda$-terms, as introduced in \cite{Zeilberger-Giorgetti},
can be described as unary-binary trees with three types of vertices
corresponding to variables (leaves), abstractions (unary nodes), and
applications (binary nodes), and admitting an injective mapping of
abstractions to variables such that each abstraction is matched to a
variable in its subtree. In particular, in each subtree the number of
variables is no less than the number of abstractions and, as
illustrated by Figure~\ref{fig:lambda-companion}, these structures can
immediately be interpreted as non-negative $Q_\lambda$-trees, with
vertex type generating function $Q_\lambda(v,w,u)=u+v^2+w$ (recall
Figure~\ref{fig:Necklaces}). 

As a consequence, the gf of planar $\lambda$-terms is governed by the corresponding catalytic
equation:
\begin{align*}
  {F}(u)=t{Q_\lambda}({F}(u),{\textstyle\frac1u}({F}(u)-F(0)),u)=tu+t{F}(u)^2+{\textstyle \frac tu}(F(u)-F(0)).
\end{align*}
Of particular interest are \emph{closed planar $\lambda$-terms} of
size $n$, which precisely correspond to non-negative
$\cQ_\lambda$-trees with excess 0 having $n$ abstraction necklaces
(and $n$ variable necklaces and $n-1$ application necklaces).
Theorem~\ref{thm:main} yields a bijection between
closed planar $\lambda$-terms with $n$ abstractions, and $n$
variables, and $n-1$ applications, and unrooted
$\cQ_\lambda$-companion trees with the corresponding necklaces.
According to Theorem~\ref{pro:CFS}, the resulting $\ws$-rooted
$\cQ_\lambda$-companion trees are governed by the context-free specification of Figure~\ref{fig:lambda-companion}, and their gf is given by the following  $\mathbb{N}$-algebraic system: 
\begin{align*}\left\{
  \begin{array}{rcl}
  C_\ws&=&tC_\ws^2+2t^2(1+C_\bc),\\
  C_\bc&=&2tC_\ws(1+C_\bc).
  \end{array}\right.
\end{align*}
Equivalently, the specification can be wrapped up in a single equation, so that $\ws$-rooted  
$\cQ_\lambda$-companion trees form a simple variety of trees in the sense of \cite{FS}, with their gf satisfying the following functional equation:
\begin{align}\label{eq:lambda-companion-tree-equation}
  C_\ws=tC_\ws^2+\frac{2t^2}{1-2tC_\ws}=\frac{2t^2}{(1-tC_\ws)(1-2tC_\ws)}.
\end{align}

Rewiring is then a bijection between the class of closed planar
$\lambda$-terms, and the class $\mathbf{C}_\lambda$ of unrooted
$\mathbf{Q}_\lambda$-companion trees (Theorems~\ref{thm:main}
and~\ref{thm:un-balanced}). 
As far as we know this is the first direct bijection between planar $\lambda$-terms and a class of unrooted simple generated trees.

Rewiring also provides a bijection
between closed planar $\lambda$-terms with at most one marked
abstraction and $\ws$-rooted $\cQ_\lambda$-companion trees
(Corollary~\ref{cor:F0pbl}).  Equivalently, since a closed $\lambda$
term of size $3n-1$ has $n$ abstractions, the latter bijection can be
rephrased as a $1$-to-$(n+1)$ correspondence between planar
$\lambda$-terms and $\ws$-rooted $\cQ_\lambda$-companion trees:
\begin{align*}
  f_n=\frac{1}{n+1}[t^{3n-1}](f+f^{\bl})=\frac1{n+1}[t^{3n-1}]C_\ws=\frac{c_n}{n+1},
\end{align*}
where the $c_n$ are given by Equation~\eqref{eq:lambda-companion-tree-equation} and satisfies $c_1=2$, $c_2=12$ and 
\[
c_n = \frac{2^4\cdot 3(3n-5)(3n-7)}{n(n-1)}\cdot c_{n-2}
\quad \textrm{for all $n\geq3$,}\]
so that
\[
c_n=\frac{2^{2n-1}(3n-3)!!}{n!(n-1)!!}
\quad\textrm{and}\quad f_n=\frac{2^{2n-1}(3n-3)!!}{(n+1)!(n-1)!!}.
\]
To the best of our knowledge this is also the first direct
relation between closed planar $\lambda$-terms and a family of
simply generated trees.

\subsection{Non-separable planar maps, two-stack-sortable permutations, fighting fish and ternary trees}

In \cite{GW} Goulden and West proposed a join catalytic recursive decomposition
of non-separable planar maps and two stack sortable
permutations. Duchi and Henriet extended this decomposition in
\cite{DuHe} to fighting fish in a way that then allowed them to give a
direct bijection between non-separable planar maps and fighting fish.
Up to a change of variable $v\to 1+u$, the underlying catalytic decomposition
yields for the generating function of non-separable planar maps
with respect to the number of edges ($t$) and pending half-edges in
the unbounded face ($u$), the following catalytic equation:
\begin{align}\label{eq:F-NS}
  {F}(u)&
  =t\cdot(1+u)\cdot(1+{F}(u))\cdot\left(1+{\textstyle \frac1u}(F(u)-F(0))\right).
\end{align}
Moreover the derivation trees naturally associated to these
decompositions are non-negative $\cQ_{ns}$-trees with
$\cQ_{ns}=(\varepsilon+\bc)\cdot(\varepsilon+\bl)\cdot(\varepsilon+\bt)$,
as given by Figure~\ref{fig:NS-nodes}. In particular, Equation~\eqref{eq:F-NS} can be written as
\begin{align*}
  F(u)&=t\cdot Q_{ns}(F(u),{\textstyle\frac1u(F(u)-F(1))},u),
\end{align*}
with $Q_{ns}(v,w,u)=(1+u)(1+v)(1+w)$ the necklace type generating function of $\mathcal{Q}_{ns}$.
\begin{figure}[t]
  \begin{center}
      \includegraphics[scale=.4]{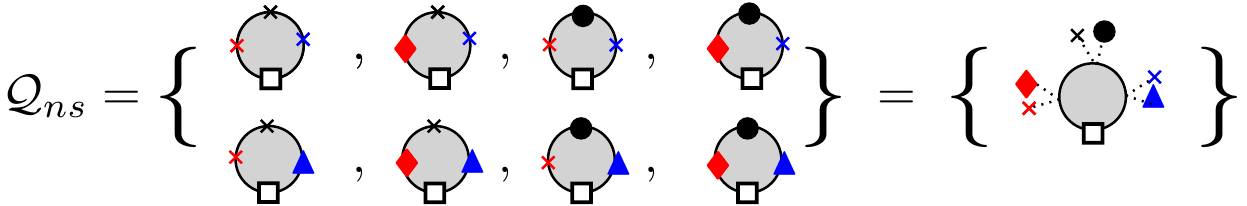}
  \end{center}
  \caption{\label{fig:NS-nodes}The 8 types of decorated necklaces of $\cQ_{ns}$}
\end{figure}
\begin{figure}[t]
  \begin{center}
      \includegraphics[scale=.4]{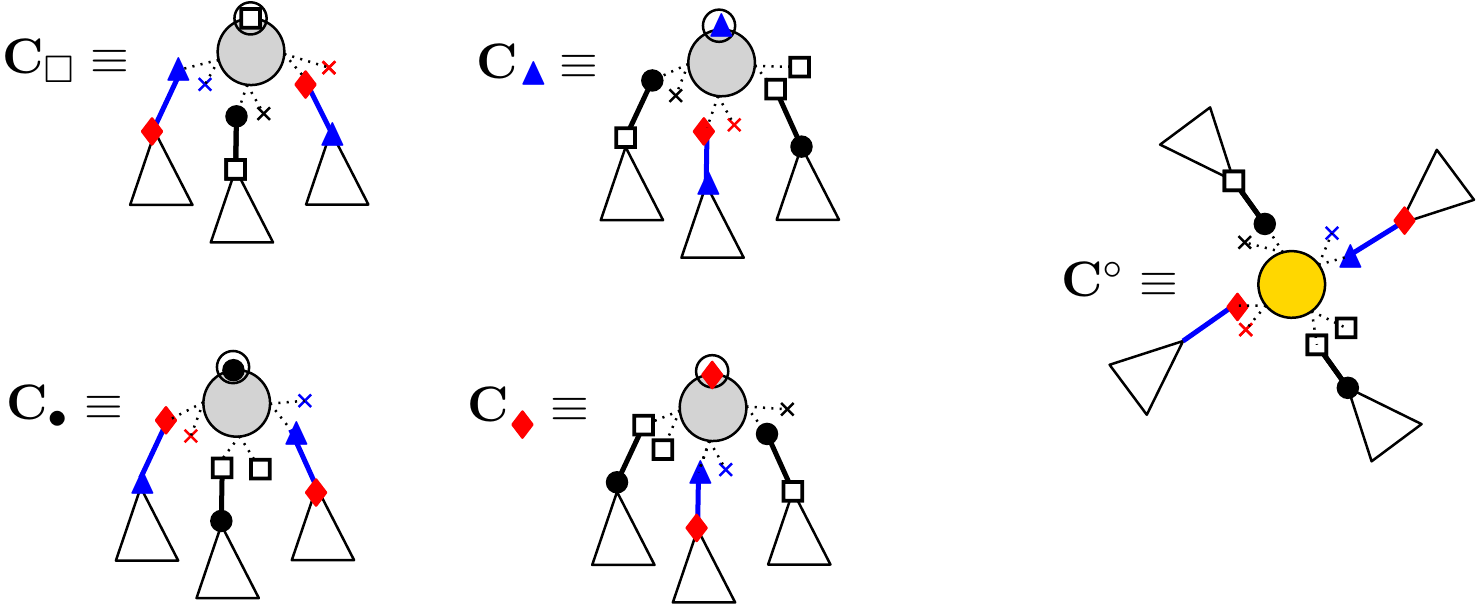} 
  \end{center}
  \caption{\label{fig:NS-derivation}Recursive decompositions of the 4 classes of rooted $\cQ_{ns}$-companion trees and of the class $\mathbf{C}^\circ$ of $\cQ_{ns}$-companion trees with a marked necklace.}
\end{figure}

Theorem~\ref{pro:CFS} then directly gives the specification of rooted $\cQ_{ns}$-companion trees as illustrated by Figure~\ref{fig:NS-derivation}, and
their generating function satisfy the
companion $\mathbb{N}$-algebraic system:
\begin{align}\label{eq:CFSfish}
  \left\{
  \begin{array}{rll}
    {C}_\ws&\!=\qquad t\cdot(1+{C}_\bl)\!\cdot\!(1+{C}_\ws)\!\cdot\!(1+{C}_\bt),\\
    {C}_\bc&\!=\qquad t\cdot(1+{C}_\bc)\!\cdot\!(1+{C}_\bl)\!\cdot\!(1+{C}_\bt),\\
    {C}_\bl&\!=\qquad t\cdot(1+{C}_\bc)\!\cdot\!(1+{C}_\bl)\!\cdot\!(1+{C}_\ws),\\
    {C}_\bt&\!=\qquad t\cdot(1+{C}_\bc)\!\cdot\!(1+{C}_\ws)\!\cdot\!(1+{C}_\bt),\\
    {C}^{\fullmoon}&\!=t\cdot (1+{C}_\bc)\!\cdot\!(1+{C}_\bl)\!\cdot\!(1+{C}_\ws)\!\cdot\!(1+{C}_\bt).
  \end{array}
  \right.
\end{align}
In particular, the first four classes are just isomorphic classes of
rooted ternary trees with deterministic colorings.  According to
Theorem~\ref{thm:main} rewiring yields a bijection between
non-negative $Q_{ns}$-trees without defects and unrooted colored
ternary trees (with deterministic 4-pearl coloring).

Following \cite{DuHe,henri2025}, the recursive
bijection between rooted non-separable planar maps and non-negative
$\mathcal{Q}_{ns}$-trees underlying Equation~\eqref{eq:F-NS} can be
realized as a direct bijection, upon using a leftmost depth first
search traversal to associate to each rooted non-separable planar map
its derivation tree in the family of non-negative
$\mathcal{Q}_{ns}$-trees.  In view of Corollary~\ref{cor:F0O}, this
then gives a direct $1$-to-$(n+1)$ correspondence between rooted
non-separable planar maps with $n$ edges and vertex-rooted colored
ternary trees with $n$ nodes, hence providing a new direct
combinatorial derivation of Tutte's formula $\frac{1}{n+1}\cdot
\frac{2}{2n+1}{3n\choose n}$ for the number of rooted non-separable
planar maps.

\subsection{Parking trees}

Non-negative $\cQ$-trees also encompass the \emph{parking trees}
introduced in \cite{Pan} and studied in
\cite{Chen,contat24, contat23,contatcurien21}. In this
context, as discussed in \cite{henri2025}, non-negative $\cQ$-trees can be viewed as a generalization of
parking trees where the $\bt$-pearls play the role of cars and the
$\bl$-pearls that of parking spots, and non-negative trees
with excess 0 correspond to so-called \emph{fully parked trees}, in
which all cars are parked at the end of the process: these trees play a
distinguished role in the study of the parking processes on random
trees.
Our rewiring bijection provides in particular a combinatorial lift of
the coupling introduced independently in \cite{contat23} to relate the
properties of a specific type of such random fully parked trees to
random Galton-Watson trees.

\subsection{Generic catalytic equations and weighted $\cQ$-trees}\label{sec:generic}
\subsubsection{Weighted trees.}
Our approach immediately extends to weighted $\cQ$-trees: 
Given a collection $(q_s)_{s\in\cQ}$ of formal weights associated to
the vertex types in $\cQ$, let the \emph{weight} of a $\cQ$-tree
$\tau$ be the product $q_\tau=\prod_{v\in\tau}q_{s(v)}$ of the weights
of the types of its vertices. Then the weighted generating function
$F(u)\equiv F(t,u)\in\mathbb{Q}[(q_s)_s][[t,u]]$
\[
F(u)=\sum_{\tau\in\mathcal{F}}q_\tau t^{|\tau|}u^{\mathrm{exc}(\tau)},
\]
satisfies Equation~\eqref{eqn:eq1}
with $Q(v,w,u)$ the weighted vertex type generating function
\[Q(v,w,u)=\sum_{s\in \cQ}q_s v^{|s|_\bc}w^{|s|_\bl}u^{|s|_\bt}.\]
Now all previously stated theorems and corollaries involve necklace-preserving bijections, hence the weights are taken into account in a
straightforward way and yield combinatorial versions of
Equations~\eqref{eqn:eq1} and~\eqref{eq:CFS}. The results also remain
true with non formal weights as soon as these weights are such that the
series $F(t,u)$ is well defined.

\subsubsection{Generic catalytic equations.}
The above weighted framework allows us to fit arbitrary catalytic
equations of the form~\eqref{eqn:eq1} within our framework: assume
that
\[Q(v,w,u)=\sum_{i,j,k}q_{i,j,k}v^iw^ju^k,\] and that the weights $q_{i,j,k}$ are such that Equation~\eqref{eqn:eq1} has a formal power series solution $F(u)\equiv F(t,u)$ in the variables $t$ and $u$ with coefficients in the same space as the weights. Then $F(u)$ is the generating series of weighted non-negative $\cQ_{\mathrm{ord}}$-trees with $\cQ_{\mathrm{ord}}=\{\bc^i\bl^j\bt^k\mid i,j,k\geq0\}$
and weights $q_{\bc^i\bl^j\bt^k}=q_{i,j,k}$, and
Formulas~\eqref{eqn:eq1} and~\eqref{eq:CFS} holds with their
interpretations in terms of weighted $\cQ_{\mathrm{ord}}$-companion trees. 

\subsubsection{Inhomogeneous catalytic specifications.}
Another application of the weighted framework is to allow to consider
more general notion of sizes, beyond the standard definition as number
of necklaces. Given a size function $\sigma:\cQ\to\mathbb{N}$ on the
set of necklaces, let the size of a $\cQ$-tree $\tau$ be defined as the
sum of the sizes of its necklaces:
$|\tau|_\sigma=\sum_{v\in\tau}\sigma(s(v))$. Then assuming that the
set of non-negative $\cQ$-trees $\tau$ with size $|\tau|_\sigma=n$ is
finite for all $n\geq0$, the corresponding gf of non-negative
$\cQ$-trees
$F_\sigma(u)=\sum_{\tau\in\mathcal{F}_{\cQ}}x^{|\tau|_\sigma}u^{\mathrm{exc}(\tau)}$
is the unique formal series solution of the equation
\begin{align}\label{eqn:inhom1}
  F_\sigma&=Q_\sigma\left(F_\sigma,\frac1u(F_\sigma-F_\sigma(0),u,x\right)
\end{align}
with
$Q(v,w,u,x)=\sum_{s\in\cQ}v^{|s|_\bc}w^{|s|_\bl}u^{|s|_\bt}x^{\sigma(s)}$,
and the gf $f$ of non-negative $\cQ$-trees with excess 0 with respect to the size is given by
\begin{align}\label{eqn:inhom2}
F(0)&=C_\ws-C_\bl\cdot C\bt, \qquad\textrm{or}\qquad  \frac{\partial}{\partial x}F(0)=(1+C_\bc)\cdot Q'_x(C_\ws,C_\bt,C_\bl,x),
\end{align}
where $C_\ws$, $C_\bc$, $C_\bl$ and $C_\bt$ are the unique solution of the equation
\begin{align}\label{eqn:inhom3}
\left\{
\begin{array}{rcl}
  C_\ws&=&Q(C_\ws,C_\bt,C_\bl,x),\\
  C_\bc&=&(1+C_\bc)\cdot Q'_v(C_\ws,C_\bt,C_\bl,x),\\
  C_\bl&=&(1+C_\bc)\cdot Q'_w(C_\ws,C_\bt,C_\bl,x),\\
  C_\bt&=&(1+C_\bc)\cdot Q'_u(C_\ws,C_\bt,C_\bl,x).\\
\end{array}
\right.
\end{align}
An example of inhomogeneous catalytic specification is the Tutte equation for rooted quasi-triangular planar maps counted
by non-root vertices and unbounded face degree \cite[Eq. (1)]{tutte}, whose gf satisfies
\begin{align}\label{eqn:triang}
T(x,u)=xu(1+uT(x,u))^2+\frac1u(T(x,u)-T_0).
\end{align}
Identifying Formula~\eqref{eqn:triang} as an instance of
Formula~\eqref{eqn:inhom1} directly yields
Parametrization~\eqref{eqn:inhom2}-\eqref{eqn:inhom3} without taking
the detour of introducing an extra variable $t$ to make the equation
homogeneous.

\subsection{Linear catalytic equations}
A catalytic equation of the form~\eqref{eqn:eq1} is said to be linear
if the series $Q(v,w,u)$ is linear in $v$ and $w$:
$Q(v,w,u)=P(u)+R(u)v+S(u)w$ with $P(u)\neq0$. Accordingly let
$\cQ=\mathcal{P}\cup\mathcal{R}\cup\mathcal{S}$ where
\begin{itemize}
\item  $\mathcal{P}$ is a non empty set of necklaces with a root $\ws$-pearl and possibly some $\bt$-pearls,
\item $\mathcal{R}$ is a set of necklaces with a root $\ws$-pearl, a $\bc$-pearl and possibly some $\bt$-pearls,
\item $\mathcal{S}$ is a set of necklaces  with a root $\ws$-pearl, a $\bl$-pearl and possibly some $\bt$-pearls.
\end{itemize}  
In this situation $\cQ$-trees have a linear structure and can be
interpreted as words on the alphabet $\cQ$. The non-negativity
condition is then similar to the standard non-negativity condition of
Dyck or more generally \L ukasiewicz words, and rewiring is
essentially a variant of the standard prefix (de)coding of plane trees
by \L ukasiewicz words \cite[Chapter 11]{lothaire} or \cite[Chapter
  6]{stanley}.

In particular, since a linear non-negative $\cQ$-tree has exactly one necklace $v\in \mathcal{P}$,  the associated unrooted $\cQ$-companion tree can be rooted at this unique necklace $v$ from $\mathcal{P}$. The root necklace decomposition then immediately yields a combinatorial derivation of the relation
\begin{align}
    f=t\cdot (1+C_\bc)\cdot P(C_\bl).
\end{align}  
Moreover in view of the restricted type of
necklace occurring in $\mathcal{R}$ and $\mathcal{S}$, the context-free
specification of Theorem~\ref{pro:CFS} yields the independent subsystem
for the classes $\mathbf{C}_\bc$ and $\mathbf{C}_\bl$,
\begin{align}
  \left\{
  \begin{array}{rcl}
    C_\bc&=&t\cdot(1+C_\bc)\cdot R(C_\bl),\\
    C_\bl&=&t\cdot(1+C_\bc)\cdot S(C_\bl).
  \end{array}
  \right.
\end{align}
In particular it is noteworthy that when $P$, $R$ and $S$ are polynomials in $\mathbb{N}[u]$ the series $f$ is $\mathbb{N}$-algebraic, as opposed to only $f'$ being $\mathbb{N}$-algebraic in the non linear case.

\section{Conclusions}\label{sec:Conclusion}

In many instances of enumeration result involving catalytic equations, \emph{e.g.}
\cite{ALCO_2020__3_2_433_0,chapoton,Chen,cori-schaeffer,FANG2020111802,GW,tutte,Zeilberger-Giorgetti},
Equation~\eqref{eqn:eq1} arises from the direct translation for the gf
$F(t,u)$ of a catalytic specification \emph{i.e.}, a combinatorial
recursive decomposition that can be put in the form
\eqref{eqn:cataspec}, for a bigraded combinatorial class $\mathcal{F}$
whose objects $\gamma$ are equipped with an additive size parameter
$|\gamma|$ with non-negative increments (marked by $t$), and an
additive catalytic parameter $c(\gamma)$ with signed increments but a
non-negativity constraint (marked by $u$): the series $f=F(t,0)$ is
the gf of the subclass $\mathcal{F}_0=\{\gamma\in\mathcal{F}\mid
c(\gamma)=0\}$ and $\frac d{dt}f$ is a gf for marked
$\mathcal{F}_0$-structures.
Following the Sch\"utzenberger methodology as described for instance
in \cite{gs:BM-ICM}, the fact that $\frac d{dt}f$ can be expressed
positively in terms of the solutions of System~\eqref{eq:CFS} raises
the question of giving a context-free specification of the
form~\eqref{eqn:combiCFS} directly for marked $\mathcal{F}_0$-structures.
To answer this question some knowledge of the actual recursive
decomposition of the $\mathcal{F}$-structures is needed, which is
typically encoded by a family of \emph{derivation trees} describing
the way the recursion unfolds.

The strength of our model of non-negative
$\mathcal{Q}$-trees is that it includes naturally many (most?) of the
derivation trees associated to first order catalytic decompositions in
the literature. As a consequence, our result can be considered as a generic recipe to
convert a catalytic specification governed by
Equation~\eqref{eqn:eq1} into a bijection between the associated
derivation trees for $\mathcal{F}_0$ and simple varieties of
multi-type trees governed by
Equations~\eqref{eq:fCFS}--\eqref{eq:CFS}. Depending on the actual
relation between the underlying combinatorial structures and their
derivation trees, this can then also lead to a direct context-free
specification of the marked $\mathcal{F}_0$-structures counted by
$\frac{d}{dt}f$.

A natural follow-up of this work is to make explicit these direct
context-free specifications and bijections with simple varieties of
trees that can be derived from our result for the various above
mentioned families of combinatorial structures governed by order-one
catalytic equations. The only results of this type we are aware of are
for planar maps, with the notable exception of \cite{fang-fusy-nadeau}.
The search for context-free specifications for planar
maps can be traced back to early work of Cori prompted by
Sch{\"u}tzenberger {\cite{AST_1975__27__1_0,cori-richard}} and has
led to a long and rich series of work {\cite{chapterMaps}} with
ongoing offsprings {\cite{bouttier2021bijective}}. Hopefully the
extension of these ideas to arbitrary structures governed by order one
catalytic equations that we propose here can lead to further
interesting developments.

\section*{Acknowledgments}
We thank Mireille Bousquet-Mélou, Wenjie Fang, Yakob Kahane and Corentin Henriet for interesting discussions and the referees of the conference on Formal Power Series and Algebraic Combinatorics for their useful suggestions of improvement on an extended abstract of this article.

\printbibliography

@article{Bonichon,
   title={On the number of planar Eulerian orientations},
   volume={65},
      DOI={10.1016/j.ejc.2017.04.009},
   journal={European Journal of Combinatorics},
   author={Bonichon, Nicolas and Bousquet-Mélou, Mireille and Dorbec, Paul and Pennarun, Claire},
   year={2017},
   month=oct, pages={59–91} }

@article{Popescu,
  title={General Néron desingularization and approximation},
  volume={104}, DOI={10.1017/S0027763000022698},
  journal={Nagoya Mathematical Journal},
  author={Popescu, Dorin},
  year={1986},
  pages={85–115}}

@inbook{appliedlothaire,
  place={Cambridge},
  series={Encyclopedia of Mathematics and its Applications},
  title={Counting, Coding, and Sampling with Words},
  booktitle={Applied Combinatorics on Words},
  publisher={Cambridge University Press},
  author={Lothaire, M.},
  year={2005},
  pages={478–519},
  collection={Encyclopedia of Mathematics and its Applications}
}

@unpublished{fang-fusy-nadeau,
      title={Tamari intervals and blossoming trees}, 
      author={Wenjie Fang and Éric Fusy and Philippe Nadeau},
      eprint={2312.13159},
      archivePrefix={arXiv},
  primaryClass={math.CO},
      year={2024},
}

@Article{tutte,
  author = 	 {Willam T. Tutte},
  title = 	 {On the enumeration of planar maps},
  journal = 	 {Bull. Amer. Math. Soc.},
  year = 	 1968,
  volume = 	 74,
  pages = 	 {64--74}}

@PHDTHESIS{henri2025,
url = "",
title = "Cartes planaires, intervalles de Tamari et arbres de stationnement : un voyage bijectif",
author = "Henriet, Corentin",
year = "2025",
}

@PHDTHESIS{notar2024,
title = "Geometry-driven algorithms for the efficient solving of combinatorial functional equations",
author = "Notarantonio, Hadrien",
year = "2024",
}

@InProceedings{S23,
  author = {Gilles Schaeffer}, 
  title = 	 {On universal singular exponents in equations with one catalytic parameter of order one},
  year = 	 2023,
  booktitle = {European Conference on Combinatorics, Graph Theory and Applications},
  address = 	 {Prague, Czech Republic},
  month = 	 {August},
  pages = 	 {pp.806-811}}

@InProceedings{DuHe,
  author = 	 {Enrica Duchi and Corentin Henriet},
  title = 	 {Combinatorics of fighting fish, planar maps and Tamari intervals},
  booktitle = {FPSAC, Bengalore. Séminaire Lotharingien de Combinatoire},
  year = 	 2022,
  number = 	 {86B.83}}

@InProceedings{BCNSED,
  author = 	 {Alin Bostan and Frédéric Chyzak and Hadrien Notarantonio and Mohab Safey El Din},
  title = 	 {Algorithms for discrete
differential equations of order 1},
  booktitle = {ISSAC},
  year = 	 2022,
  month = 	 {Jul},
  address = 	 {Lille, France}}

@Misc{Chen,
  key = 	 {Enumeration of fully parked trees},
  author = 	 {Linxiao Chen},
  title = 	 {Enumeration of fully parked trees},
  year = 	 2021,
  eprint={2103.15770},
  archivePrefix={arXiv},
  primaryClass={math.CO}
}

@book{FS,
	Author = {Flajolet, Philippe  and Sedgewick, Robert },
	Publisher = {Cambridge University Press},
	Title = {Analytic combinatorics},
	Year = {2009}}

@inproceedings{gs:BM-ICM,
  author = 	 {Bousquet-M\'elou, Mireille},
  title = 	 {Rational and algebraic series in combinatorial enumeration},
  booktitle = {Proceedings of the Internatinal Congress of Mathematicians},
  year = 	 2006}

@article{BMJ,
  author    = {Mireille Bousquet{-}M{\'{e}}lou and
               Arnaud Jehanne},
  title     = {Polynomial equations with one catalytic variable, algebraic series
               and map enumeration},
  journal   = {J. Comb. Theory, Ser. {B}},
  volume    = {96},
  number    = {5},
  pages     = {623--672},
  year      = {2006},
  bibsource = {dblp computer science bibliography, https://dblp.org}
}

@article{FANG2020111802,
title = {A partial order on Motzkin paths},
journal = {Discrete Math.},
volume = {343},
number = {5},
pages = {111802},
year = {2020},
author = {Wenjie Fang},
}

@article{DNY,
title = {Universal singular exponents in catalytic variable equations},
journal = {Journal of Combinatorial Theory, Series A},
volume = {185},
pages = {105522},
year = {2022},
author = {Michael Drmota and Marc Noy and Guan-Ru Yu},
}

@article{Zeilberger-Giorgetti,
  author    = {Noam Zeilberger and
               Alain Giorgetti},
  title     = {A correspondence between rooted planar maps and normal planar lambda
               terms},
  journal   = {Log. Methods Comput. Sci.},
  volume    = {11},
  number    = {3},
  year      = {2015},
  bibsource = {dblp computer science bibliography, https://dblp.org}
}

@InCollection{chapterMaps,
  author = 	 {Schaeffer, Gilles},
  title = 	 {Planar Maps},
  booktitle = 	 {Handbook of Enumerative Combinatorics},
  publisher = {Chapman and Hall/CRC},
  year = 	 2015,
  chapter = 	 {V},
  edition = 	 {1st}}

@article{bouttier2021bijective,
     author = {Bouttier, J\'er\'emie and Guitter, Emmanuel and Miermont, Gr\'egory},
     title = {Bijective enumeration of planar bipartite maps with three tight boundaries, or how to slice pairs of pants},
     journal = {Annales Henri Lebesgue},
     pages = {1035--1110},
     publisher = {\'ENS Rennes},
     volume = {5},
     year = {2022},
     doi = {10.5802/ahl.143},
}

@book{AST_1975__27__1_0,
     author = {Cori, Robert},
     title = {Un code pour les graphes planaires et ses applications},
     series = {Ast\'erisque},
     publisher = {Soci\'et\'e math\'ematique de France},
     number = {27},
     year = {1975},
     zbl = {0313.05115},
     mrnumber = {404045},
}

@article{ALCO_2020__3_2_433_0,
     author = {Chapoton, Fr\'ed\'eric},
     title = {Some properties of a new partial order {on~Dyck} paths},
     journal = {Algebraic Combinatorics},
     pages = {433--463},
     publisher = {MathOA foundation},
     volume = {3},
     number = {2},
     year = {2020},
}

@Article{chapoton,
  author = 	 {Chapoton, Frédéric},
  title = 	 {Sur le nombre d'intervalles dans les treillis de Tamari},
  journal = 	 {Sém. Lothar. Combin.},
  year = 	 2005,
  volume = 	 55,
  number = 	 {B55f}
}

@article{GW,
  title={Raney paths and a combinatorial relationship between rooted nonseparable planar maps and two-stack-sortable permutations},
  author={Goulden, Ian P and West, Julian},
  journal={Journal of Combinatorial Theory, Series A},
  volume={75},
  number={2},
  pages={220--242},
  year={1996},
}

@article{cori-schaeffer,
title = {Description trees and Tutte formulas},
journal = {Theoretical Computer Science},
volume = {292},
number = {1},
pages = {165-183},
year = {2003},
note = {Selected Papers in honor of Jean Berstel},
author = {Robert Cori and Gilles Schaeffer}
}

@article{cori-richard,
title = {Enumération des graphes planaires à l'aide des séries formelles en variables non commutatives},
journal = {Discrete Math.},
volume = {2},
number = {2},
pages = {115-162},
year = {1972},
author = {Robert Cori and Jean Richard}
}

@misc{Contat23,
      title={Parking on trees with a (random) given degree sequence and the Frozen configuration model}, 
      author={Alice Contat},
      year={2023},
      eprint={2312.04472},
      archivePrefix={arXiv},
      primaryClass={math.PR},
}

@Article{contat24,
  author = 	 {Alice Contat},
  title = 	 {Last car decomposition of planar maps},
  journal = 	 {Ann. Inst. Henri Poincaré Comb. Phys. Interact.},
  year = 	 2024}

@Article{contatcurien21,
  author = 	 {Alice Contat and Nicolas Curien},
  title = 	 {Parking on Cayley trees \& Frozen Erdős-Rényi},
  journal = 	 {The Annals of Probability},
  year = 	 2021}

@article{Pan,
title = {Parking functions for mappings},
journal = {Journal of Combinatorial Theory, Series A},
volume = {142},
pages = {1-28},
year = {2016},
author = {Marie-Louise Lackner and Alois Panholzer}
}

@Book{stanley,
  author = 	 {Richard P. Stanley},
  title = 	 {Enumerative combinatorics, Vol. 2},
  publisher = 	 {Cambridge University Press},
  year = 	 1999}

@book{lothaire,
author = {M. Lothaire},
place={Cambridge},
edition={2},
series={Cambridge Mathematical Library},
title={Combinatorics on Words},
year={1997}
}

\end{document}